\documentclass[runningheads]{llncs}
\usepackage[top=80pt,bottom=80pt,left=80pt,right=80pt]{geometry}
\usepackage[utf8]{inputenc}
\usepackage[T1]{fontenc}
\usepackage{graphicx}

\usepackage{amssymb}
\usepackage{amsmath}
\usepackage{dsfont}

\usepackage{enumitem}
\usepackage{natbib}
\usepackage{tikz}
\usepackage{hyperref}
\usepackage[capitalize,nameinlink,noabbrev]{cleveref}

\usepackage[dvipsnames]{xcolor}

\usepackage{upgreek}
\usepackage{comment}
\usepackage{float}
\usepackage{cancel}
\usepackage{algorithm}
\usepackage{algpseudocode}
\usepackage[autostyle]{csquotes}

\usepackage[normalem]{ulem}
\useunder{\uline}{\ul}{}

\usetikzlibrary{graphs}
\usetikzlibrary{graphs.standard}
\usetikzlibrary{arrows,decorations.markings}
\tikzset{
    mid arrow/.style={
        line width=0.75pt, 
        postaction={
            decorate,
            decoration={
                markings,
                mark=at position 0.5 with {\arrow{>}}
            }
        }
    }
}

\usetikzlibrary{tikzmark,shapes.misc}

\newtheorem{obs}{Observation}
\newtheorem{PB}{Open Problem}

\def\mG{\mathcal{G}}
\DeclareEmphSequence{\bfseries\itshape}

\DeclareMathOperator{\Ness}{Ness}

\newcommand{\bfit}[1]{\textbf{\textit{#1}}}

\newcommand{\toremove}[1]{\textcolor{orange}{}}

\title{Extending Ghouila-Houri's Characterization of Comparability Graphs to Temporal Graphs%
\thanks{This work was supported by the French ANR projects ANR-24-CE48-4377 (GODASse) and ANR-22-CE48-0001 (TEMPOGRAL).}}
\date{}
\author{Pierre Charbit \and
Michel Habib  \and
Amalia Sorondo}
\institute{}
\institute{Université Paris Cité, CNRS, IRIF, F-75013, Paris, France 
\\
\email{\{charbit,habib,sorondo\}@irif.fr}}
\begin{document}
\maketitle
\begin{abstract}
An orientation of a static graph is called \bfit{transitive} if for any three vertices $a,b,c$, the presence of arcs $(a,b)$ and $(b,c)$ forces the presence of arc $(a,c)$. If only the presence of an arc between $a$ and $c$ is required, but its orientation is unconstrained, the orientation is called \bfit{quasi-transitive}. A fundamental result due to Ghouila-Houri \cite{GhouilaHouri} states that any static graph admitting a quasi-transitive orientation also admits a transitive orientation. In a seminal work \cite{mertzios2025complexity}, Mertzios et al. introduced the notion of temporal transitivity in order to model information flows in simple temporal networks. We revisit the model introduced by Mertzios et al. and propose an analogous to Ghouila-Houri's characterization for the temporal scenario. We present a structural theorem that will allow us to express by a 2-SAT formula all the constraints imposed on a temporal graph for it to admit a temporal transitive orientation. The latter produces an efficient recognition algorithm for graphs admitting such orientations, that we will call comparability temporal graphs. Inspired by the lexicographic strategy presented by Hell and Huang in \cite{hell1995lexicographic} to transitively orient static graphs, we then propose an algorithm for constructing a temporal transitive orientation of a YES instance. This algorithm is straightforward and has a running-time complexity of $O(nm + \min\{kn,m^2\})$, with $n$, $m$ and $k$ being respectively the number of vertices, edges and monolabel triangles, i.e., triangles having the same unique time-label on their edges, in the temporal graph. This represents an improvement compared to the algorithm presented in \cite{mertzios2025complexity}. Additionally, we extend the temporal transitivity model to temporal graphs having multiple time-labels associated to their edges and claim that the previous results hold in the multilabel setting. Finally, we propose a characterization of comparability temporal graphs by forbidden temporal ordered patterns.  
\end{abstract}

\section{Introduction}
\label{sect:intro}

\paragraph*{Context.} In the static (i.e., not temporal)  case, an undirected graph is a \bfit{comparability} graph if and only if it admits a \bfit{transitive orientation}, that is, if it is possible to orient its edges such that, whenever an edge $ab$ is oriented from $a$ towards $b$ and an edge $bc$ from $b$ towards $c$ (i.e., arcs $(a,b)$ and $(b,c)$ belong to the orientation), a third edge $ac$ must exist and be oriented from $a$ towards $c$. If only the existence of the third edge $ac$ is required, but its orientation is unconstrained, the orientation is called \bfit{quasi-transitive}. 

Consider a graph $G=(V,E)$ and a quasi-transitive orientation $O$ of $G$.  Whenever $ab \in E$ and $bc \in E$ but $ac \notin E$, orienting $ab$ from $a$ towards $b$ in $O$ will force edge $bc$ to be oriented from $c$ towards $b$ (and analogously, orienting $ab$ from $b$ towards $a$ will force the orientation of $bc$ from $b$ towards $c$). Thus, by defining a Boolean variable $x_{ab}$ for each edge $ab \in E$ and interpreting the truth assignment 1 to the variable as the orientation from $a$ towards $b$, and the truth assignment 0 as the opposite orientation, these forcing relations between arcs, i.e., the presence of an arc in the orientation implying the presence of a second one, can be expressed by clauses of arity two of the type $(x_{ba} \lor x_{cb}) $. Hence, deciding if a graph admits a quasi-transitive orientation can be reduced to solving a $2$-SAT problem, which can be solved by the famous Tarjan's algorithm \cite{aspvall1979linear} in time linear in the size of the set of clauses, which, for our problem,  corresponds to the number of induced paths on three vertices. Using the same approach to determine the existence of a transitive orientation is less efficient, because forbidding a directed triangle requires adding two clauses of size three and therefore obtaining an equivalent 3-SAT problem. In this context, the following result proposed by Ghouila-Houri is particularly interesting as it allows the recognition problem of comparability graphs to be solved by 2-SAT in $O(nm)$, with $n$ and $m$ being respectively the number of vertices and edges of the graph.

\begin{theorem}\cite{GhouilaHouri}\label{theo:GH}
A graph $G$ admits a transitive orientation if and only if it admits a quasi-transitive orientation.
\end{theorem}

A first  english version of \cite{GhouilaHouri} can be found in the Appendix section of this article. Note that Theorem \ref{theo:GH} does not imply that every quasi-transitive orientation is transitive, since a quasi-transitive orientation may contain directed triangles, and therefore does not solve the problem of obtaining such an orientation.
How to efficiently obtain a transitive orientation of a given comparability graph was extensively studied. We will present a brief survey of the algorithms proposed to solve the problem.

In his proof, published in the early 1960s, Ghouila-Houri implicitly formulates an algorithm for constructing a transitive orientation of a comparability graph from a quasi-transitive orientation. Given a static graph $G=(V,E)$ and a quasi-transitive orientation of its edges $O$, the strategy is to break the cyclical orientation of triangles in $O$ by identifying a module $M$ containing some, or all, of the triangle vertices, transitively orienting the module, recursively orienting the subgraph induced by $V$-$M$ and orienting all edges between $M$ and $V$-$M$ in the same direction. This approach laid the groundwork for subsequent techniques using modular decomposition, as we will see below.

In \cite{Golumbic2nd}, Golumbic proposed an algorithm for constructing a transitive orientation of a comparability graph. The algorithm is based on the observation, mentioned earlier, that the constraints associated to transitive orientations can be expressed as local forcings between the possible edge orientations. Since these forcings are symmetric, one can partition the possible orientations of the edges of a graph into equivalence classes, called  \bfit{implication classes}. Then, Golumbic's algorithm will construct the orientation by selecting an arbitrary unoriented edge, fixing its orientation, propagating it through its implication class, and then repeating the procedure on the subgraph induced by the not yet oriented edges. Note that the implication classes are re-computed at each iteration as only unoriented edges will be used to identify forcings, this is the key to avoid orienting triangles cyclically. The algorithm proposed by Golumbic has a running-time of $O(nm)$. Later on, Spinrad \cite{spinrad1985comparability} proposed an algorithm with an $O(n^2)$ running-time, relying on a modular decomposition technique based on the same principles proposed by Ghouila-Houri in his proof. Along the same lines, in the late 1990s, McConnell and Spinrad \cite{mcconnell1999modular} proposed linear time algorithms based on efficiently computing the modular decomposition of a given graph.

In his seminal article \cite{GhouilaHouri}, Ghouila-Houri proposed another important characterization of comparability graphs. Let $G=(V,E)$ be a static graph and $G^+$ the graph obtained by taking as vertices both possible orientations $(a,b)$ and $(b,a)$ for each edge $ab \in E$, and by connecting all vertices $(a,b)$ to $(b,a)$ and to all $(b,c)$ such that $ac \notin E$. One can think of $G^+$ as the \enquote{conflict} graph of $G$ as it links incompatible edge orientations. Then, it is easy to see that $G$ admits a transitive orientation if and only if $G^+$ admits a 2-coloring of its vertices. Once again, this characterization only provides a straightforward recognition algorithm,  leaving unsolved the problem of obtaining a transitive orientation of $G$ in the case of a positive instance. In the 1990s, Hell and Huang \cite{hell1995lexicographic} used the previous characterization and provided a lexicographic argument to greedily 2-color the components of $G^+$ in order to obtain a transitive orientation of a comparability graph. The technique, as well as the correctness proof, are very straightforward and result in an algorithm having $O(nm)$ time complexity.

Although obtaining a transitive orientation can be done in linear time if the graph was a comparability graph, up to our knowledge there is no linear time algorithm to verify if a given orientation is transitive, or even quasi-transitive, as this problem is closely related to the Boolean Matrix Multiplication Problem. Therefore, the recognition of comparability graphs still requires $O(nm)$ running-time in the worst case.

\paragraph*{Temporal Comparability Graphs.}
Temporal (or time-varying, or evolving) graphs were introduced with the aim of modeling dynamic networks, that is, networks whose connections evolve over time. They are of great use in modeling a large number of real-world systems, such as information, traffic and social networks, see \cite{casteigts2012time, latapy2018stream}. Our work uses the formalism presented in the foundational paper of Kempe et al.\cite{kempe2000connectivity}, where a \bfit{temporal network} is defined as a pair $(G,\lambda)$, with $\lambda$ a time-labelling function over the edges of a static graph $G$, specifying the discrete time at which its endpoints interact. With the objective of modeling the flow of information on paths whose time-labels respect the ordering of time, Kempe et al. introduced the notion of \bfit{(strict) time-respecting path}, where edges $e_1, e_2, ..., e_k$ of a given path in $G$ satisfy $\lambda(e_1) \leq \lambda(e_2) \leq ... \leq \lambda(e_k)$ (resp. $\lambda(e_1) < \lambda(e_2) < ... < \lambda(e_k)$). As in these paths the edges model undirected interactions between the network nodes, Mertzios et al.\cite{mertzios2025complexity} explored directed communications, extending the classical notion of transitive orientations to the temporal scenario. As different versions of temporal transitivity were previously considered in diverse areas such as medical data treatment \cite{moskovitch2015fast} or text processing \cite{tannier2011evaluating}, Mertzios et al. propose a new definition which can be justified in the context of confirmation and verification of information in a temporal network. One can consider a scenario where an important information is sent from node $a$ to node $b$ at time $t_1$ and then this intermediary node sends the information to node $c$ at $t_2$, with $t_1 \leq t_2$. Subsequently, node $c$ might want to verify the information by querying directly from node $a$ at a time $t_3 \geq t_2$. Before formally defining temporal transitive orientations, let us make a parallelism with the static case by introducing quasi-temporal transitive orientations. Whenever a directed time-respecting path $abc$ is formed, temporal transitive orientations will ask for a third edge $ac$ to exist, to have a greater or equal time-label than the one associated to edge $bc$ and to be oriented from $a$ towards $c$. As for quasi-transitive orientations in the static scenario, the orientation of this third arc will not be restricted by quasi-temporal transitive orientations, unlike its existence and time-label value. 

\begin{definition}\label{def:TTO}
An  orientation $O$ of a temporal graph $\mathcal{G} = (G,\lambda)$ is a temporal transitive orientation, or \bfit{TTO}, (resp. \bfit{QTTO}) if whenever $(a, b)\in O $ and $(b, c) \in O$, with $\lambda(ab) \leq \lambda(bc)$ then $(a, c) \in O$ (resp. $(a, c)$ or $(c, a) \in O)$ with $\lambda(ac) \geq \lambda(bc)$.
\end{definition}

Note that the time-label value of an edge is independent of its orientation. Then, by the previous definition, we observe that only monolabel triangles, i.e., triangles which have a unique time-label associated to their three edges, can be cyclically directed in a QTTO. This will be the equivalent to directed triangles in the static case, as it corresponds to the only restriction imposed by temporal transitive orientations that cannot be captured by Boolean clauses of arity two. We will refer to temporal graphs admitting a temporal transitive orientation as \bfit{temporal comparability graphs}. One can note that this definition is a true generalization of comparability graphs in the static case, since if all edges share the same time-label, then $\mathcal{G} = (G,\lambda)$ is a comparability temporal graph if and only if $G$ is a comparability graph. Given a temporal graph $\mathcal{G}$, the TTO Problem consists of deciding whether $\mathcal{G}$ admits a temporal transitive orientation of its edges. Mertzios et al. presented an algorithm to solve this problem, inspired by the strategy proposed by Golumbic \cite{Golumbic2nd} for the static case. This way, they extend the forcing relation to consider the temporal labels of the edges and express the temporal transitive constraints by a Boolean formula formed by the conjunction of a 3-NAE and a 2-SAT formula. Recall that 3-NAE stands for 3-Not-All-Equal, and is formed by a conjunction of clauses, with three literals each, satisfied when at least one of the literals receives 1 as truth assignment and another one 0. The goal of these clauses is to capture the constraint of orienting monolabel triangles in a non-cyclical way. As the problem of 3-NAE is NP-complete, the authors use structural arguments specific to the TTO Problem to prove that the overall running-time of their algorithm is polynomial (see \cite{mertzios2025complexity} for more details), although they do not explicitly state its complexity and the correctness proof is rather technically involved. Different variants to the problem were defined by Mertzios et al. Given an orientation $O$ of a temporal graph $\mathcal{G}$, if we only ask for $(a,c) \in O$ with $\lambda(ac) \geq \lambda(bc)$ whenever $abc$ is a strict time-respecting directed path, i.e.,  $\lambda(ab) < \lambda(bc)$, then the problem is defined as \textsc{Strict TTO}. Similarly, if one asks the transitive arc from $a$ to $c$ to have a time-label such that $\lambda(ac) > \lambda(bc)$, the orientation is called strongly temporal transitive, and the associated problem \textsc{Strong TTO}. By combining the two requirements, one gets the \textsc{Strong Strict TTO} problem. Mertzios et al. proved in \cite{mertzios2025complexity} that the \textsc{TTO}, \textsc{Strong TTO} and \textsc{Strong Strict TTO} problems can be solved in polynomial time while deciding if a temporal graph admits a Strict TTO is \textsc{NP}-hard.
\vspace{-1ex}
\paragraph*{Forbidden Patterns Characterizations.}
Hereditary classes of graphs, such as comparability graphs, can be characterized by their set of minimal obstructions for the induced subgraph relation as can be seen in \cite{Diestel12}. Very often, these sets are unknown, or known but infinite. The use of additional structures, such as an ordering of the graph's vertices, can allow to describe these properties by finite sets of forbidden structures. For example, it is not difficult to see that comparability graphs are exactly the graphs for which there exists an ordering $\prec$ of their vertex set such that there is no ordered triple $x\prec y \prec z$ where $xy$ and $yz$ are edges of the graph but $xz$ is not. The study of characterizations by such forbidden ordered patterns was proposed by Damaschke \cite{damaschke1990forbidden}. Since then, multiple studies have been carried out, allowing the description of many hereditary graph classes. In \cite{hell2014ordering}, Hell et al. show that all graph classes characterized by sets of three-vertex ordered patterns can be recognized in polynomial time. In \cite{FeuilloleyH21}, the authors refine the previous analysis presenting a detailed characterization of all 24 graph classes that can be described by a set of three-vertex ordered patterns, showing that all classes except two of them can be recognized in linear time. In \cite{csikos2025forbidden}, they introduce the usage of forbidden ordered patterns for temporal graphs. In the temporal case, a pattern is a temporal subgraph that represents a specific arrangement with respect to the order not only of its vertices but also of its edges based on their time-labels.

\paragraph*{Our contributions and structure of the paper.} In Section \ref{sect:preliminaries}, we present the main notations and definitions that will be used throughout the paper. Then, we begin Section \ref{sect:tto-algorithm} by observing that Ghouila-Houri's result, stated in Theorem \ref{theo:GH}, is not directly adaptable to the temporal case, since it is not true that every temporal graph admitting a QTTO admits a TTO, as shown by the example in Figure \ref{fig:qtto-not-tto}. Based on this observation, we define almost-temporal transitive orientations, or ATTO, which are slightly more restrictive than the QTTO ones. 
Then, we establish the exact analogue of Ghouila-Houri's theorem by proving that a temporal graph $\mathcal{G}$ admits an ATTO if and only if it admits a TTO. Inspired by the lexicographic strategy presented by Hell and Huang in \cite{hell1995lexicographic}, we present a lexicographic simple approach to solve the TTO problem with a $O(nm + \min\{kn,m^2\})$ running-time, with $k$ being the number of monolabel triangles in the temporal graph. As most temporal networks allow resources to traverse their arcs at multiple given times and therefore result in multilabel temporal graphs, we finish this section by extending the temporal transitivity notion to the multilabel setting and stating that the proposed results hold as well. In Section \ref{sect:patterns}, we introduce the notations we will use to work with temporal forbidden ordered patterns and we provide a characterization for the class of temporal graphs admitting a temporal transitive orientation by means of these structures, both for TTO and Strict TTO. We conclude in Section \ref{sect:conclusion} by presenting a series of open problems, suggesting interesting directions for future research.

\section{Preliminaries}

\subsection{Definitions and notations}
\label{sect:preliminaries}

A \bfit{graph} is a pair $G=(V,E)$ where $V$ is the vertex set and $E \subseteq V^2$ is the set of edges. Usually,  for algorithmic complexity analysis, we take $n=|V|$ and $m=|E|$. If we consider $G=(V,A)$ with $A$ a set of ordered pairs of vertices, called arcs, we say $G$ is a \bfit{directed graph} or a \bfit{digraph}. Note that a directed graph might contain 2-cycles, i.e., bidirected edges. If $u$ and $v$ are vertices of a graph (resp. digraph), we denote by $uv$ the edge between $u$ and $v$ (resp. $(u,v)$ the arc from $u$ to $v$). In the context of directed graphs, we will abuse the notation to denote as a \bfit{directed cycle of size $k$} the set of vertices $\{v_1, \dots, v_k\}$ such that $v_iv_{i+1}$,  for all $1\leq i < k$, and $v_kv_1$ are arcs.
We will refer to the directed cycle of size three as a \bfit{directed triangle}.

A \bfit{temporal graph} is a graph whose vertex set is fixed while its edge set changes over time. Given a classical graph $G=(V,E)$, we can obtain a temporal graph $\mathcal{G}$ by assigning a set of time-labels to its edges such that $\mathcal{G}=(G,\lambda)$ with $\lambda:E\rightarrow 2^{\mathds{N}}$. These time-labels indicate the discrete time steps in which a given edge is active. We say $G$ is the \bfit{underlying graph} of $\mathcal{G}=(G,\lambda)$. A \bfit{temporal subgraph} $(G',\lambda')$ of a temporal graph $(G,\lambda)$ is a temporal graph such that $G'=(V',E')$ is a subgraph of $G$ and $\lambda'(e) \subseteq \lambda(e)$ for every $e \in E'$. A temporal graph is \bfit{simple} if $\lambda: E\rightarrow \mathds{N}$, that is, every edge has a single presence time. As we will mostly reference simple temporal graphs, we will refer to them as temporal graphs and when considering temporal graphs admitting multiple time-labels per edge, we will explicitly call them \bfit{multilabel} temporal graphs. We say that a temporal graph is \bfit{monolabel} if all of its edges share the same single time-label. In the temporal scenario, a classic graph might be referred to as a \bfit{static graph}. We say two temporal graphs $\mathcal{G}=(G,\lambda_{G})$ and $\mathcal{H}=(H,\lambda_{H})$ are \bfit{isomorphic} if there exists a bijection $\phi$ between their vertex sets such that any two vertices $u,v$ are adjacent in $G$ if and only if $\phi(u), \phi(v)$ are adjacent in $H$ and if $uv $ is an edge then  $\lambda_G(uv) = \lambda_H(\phi(u)\phi(v))$.

Following \cite{GhouilaHouri, hell1995lexicographic}, an \bfit{orientation} $O$ of an undirected graph $G=(V,E)$ consists in the choice of exactly one of the pair $(a, b)$, $(b, a)$ for each edge $ab \in E$. We refer to a graph equipped with an orientation as an \bfit{oriented graph}. Let us denote by $AP(G)$ the set of all possible arcs in an orientation of $G$, that is $\{(u,v) \hspace{0.15ex} | \hspace{0.15ex}  uv\in E \text{ or } vu \in E\}$. For a subset $O$ of $AP(G)$ we denote by its \bfit{reverse} the set defined by $O^{R} = \{(v,u) \hspace{0.25ex} | \hspace{0.15ex} (u,v)\in O\}$. A \bfit{partial orientation} of $G$ is a subset $O\subset AP(G)$ satisfying $O \cap O^{R} = \emptyset$. An \bfit{orientation} is thus a partial orientation $O$ satisfying  $O \cup O^{R} = AP(G)$.

\subsection{Preliminary observations on  temporal comparability graphs}\label{subsect:preliminar-properties}%

Let $\mathcal{G}=(G,\lambda)$ be a temporal graph and $O$ a TTO of $\mathcal{G}$. Observe that, unlike for the static case, $O^{R}$ is not necessarily a TTO of $\mathcal{G}$. We can take the odd cycle in Figure \ref{fig:t-comparability-examples} $(ii)$ as an example. Of course, not all temporal graphs are comparability temporal graphs, as we can take any non-comparability static graph and assign a unique time-label to its edges. What is more, for a non comparability temporal graph $\mathcal{G}=(G,\lambda)$, the underlying graph $G$ might be a comparability graph, as shown in Figure \ref{fig:t-comparability-examples} $(i)$. The opposite situation being also possible, i.e. a non-comparability static graph as underlying graph of a comparability temporal graph, as we see in Figure \ref{fig:t-comparability-examples} $(ii)$.

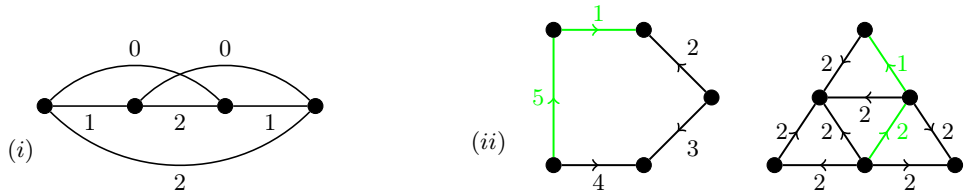
\begin{figure}[H]
    \begin{center}
        \begin{tikzpicture}[scale=0.6]
        \node (a) at (0, 0) [circle, draw, inner sep=2pt, fill=black] {};
        \node (b) at (2, 0) [circle, draw, inner sep=2pt, fill=black] {};
        \node (c) at (4, 0) [circle, draw, inner sep=2pt, fill=black] {};
        \node (d) at (6, 0) [circle, draw, inner sep=2pt, fill=black] {};
        \draw[-,semithick] (a) -- (b) node[midway, below] {1};
        \draw[-,semithick] (c) -- (b) node[midway, below] {2};
        \draw[-,semithick] (c) -- (d) node[midway, below] {1};
        \draw[-,semithick, bend right=45] (c) to node[midway, above] {0} (a);
        \draw[-,semithick, bend left=45] (b) to node[midway, above] {0} (d);
        \draw[-,semithick, bend left=45] (d) to node[midway, below] {2} (a);
        \node[text width=1cm] (I) at (0,-1) {$(i)$};
        \end{tikzpicture} \hspace{12ex}
        \begin{tikzpicture}[scale=0.6]
        \node[draw, circle, inner sep=2pt, fill=black] (A) at (1,1.5) {};
        \node[draw, circle, inner sep=2pt, fill=black] (B) at (1,-1.5) {};
        \node[draw, circle, inner sep=2pt, fill=black] (C) at (3,1.5) {};
        \node[draw, circle, inner sep=2pt, fill=black] (D) at (3,-1.5) {};
        \node[draw, circle, inner sep=2pt, fill=black] (E) at (4.5,0) {};
        \draw[mid arrow, green] (A) -- (C) node[midway, above] {1};
        \draw[mid arrow] (E) -- (C) node[midway, above right] {2};
        \draw[mid arrow] (B) -- (D) node[midway, below] {4};
        \draw[mid arrow, green] (B) -- (A) node[midway, left] {5};
        \draw[mid arrow] (E) -- (D) node[midway, below right] {3};
        \node[text width=1cm] (I) at (0,-1) {$(ii)$};
        \end{tikzpicture} \hspace{3ex}
        \begin{tikzpicture}[scale=0.6]
        \node[draw, circle, inner sep=2pt, fill=black] (A) at (0,0) {};
        \node[draw, circle, inner sep=2pt, fill=black] (B) at (1,1.5) {};
        \node[draw, circle, inner sep=2pt, fill=black] (C) at (2,0) {};
        \node[draw, circle, inner sep=2pt, fill=black] (E) at (-1,-1.5) {};
        \node[draw, circle, inner sep=2pt, fill=black] (F) at (3,-1.5) {};
        \node[draw, circle, inner sep=2pt, fill=black] (D) at (1,-1.5) {};
        \draw[mid arrow] (B) -- (A) node[midway, left] {2};
        \draw[mid arrow, green] (C) -- (B) node[midway, right] {1};
        \draw[mid arrow] (C) -- (A) node[midway, below] {2};
        \draw[mid arrow] (E) -- (A) node[midway, left] {2};
        \draw[mid arrow] (C) -- (F) node[midway, right] {2};
        \draw[mid arrow] (D) -- (F) node[midway, below] {2};
        \draw[mid arrow] (D) -- (E) node[midway, below] {2};
        \draw[mid arrow] (D) -- (A) node[midway, left] {2};
        \draw[mid arrow, green] (D) -- (C) node[midway, right] {2};
        \end{tikzpicture}
    \end{center}
    \caption{$(i)$ Comparability graph $K_4$ as underlying static graph of a non-comparability temporal graph. $(ii)$ Non-comparability graphs as underlying static graphs of comparability temporal graphs.}
    \label{fig:t-comparability-examples}
\end{figure}

However, one can prove that every static graph $G$ admits a transitive temporalization, that is, the assignation of time-labels to its edges resulting in a comparability temporal graph. It is sufficient to consider an arbitrary ordering $\sigma$ of the vertices of $G$ and defining $\lambda$ as:

$$
\lambda(v_{i}v_{j})=
\begin{cases}
n-i \hspace{1.5ex} \text{if} \hspace{1.5ex} v_{i} \prec_{\sigma} v_{j}\\
n-j \hspace{1.5ex} \text{otherwise}
\end{cases}
\hspace{1ex} \text{for every} \hspace{1ex} v_{i}v_{j} \in E \hspace{1ex} \text{with} \hspace{1ex} i,j \in \{1, ..., n \}, i \neq j
$$

If we define $O$ by orienting every edge from left to right according to $\sigma$, all directed paths in $O$ will have strictly decreasing time-labels in $\mathcal{G}$. Therefore, all arcs will be unconstrained and the temporal transitivity conditions will be trivially satisfied. Moreover, we can verify that $O$ is not only a TTO of $\mathcal{G}$ but a Strict TTO, Strong TTO and Strong Strict TTO as well. We also note the existence of static graphs for which any temporal labelling will produce a comparability temporal graph, such is the case for comparability triangle-free graphs, such as trees and even cycles.

\section{Recognition of comparability temporal graphs and TTO construction}
\label{sect:tto-algorithm}


\subsection{Structural characterization of temporal transitivity}

As explained in the introduction, in the static setting, the key to deciding whether a transitive orientation exists, or to construct one, is the weaker notion of quasi-transitive orientations. Although not every quasi-transitive orientation is transitive, Theorem~\ref{theo:GH} shows that the existence of a quasi-transitive orientation implies the existence of a transitive one. All known algorithms rely, either implicitly or explicitly, on this fact.

In the temporal case, Definition \ref{def:TTO}, introduced by Mertzios et al., provides a generalization of both transitive and quasi-transitive orientations (TTO and QTTO). As in the static case, any TTO  is also a QTTO, and again the converse is not true since a cyclic orientation of a monolabel triangle is a QTTO but not a TTO. The crucial difference is that Theorem \ref{theo:GH} does not extend with these definitions of TTO and QTTO, as witnessed by the example shown in Figure \ref{fig:qtto-not-tto}: the temporal graph is oriented by a QTTO but one can prove it does not admit any TTO (a short proof can be derived using Lemma \ref{lem:necessary} below).

\begin{figure}[H]
\begin{center}
\begin{tikzpicture}[scale=0.6]

\node[draw, circle, inner sep=1pt] (A) at (0,0) {$a$};
\node[draw, circle, inner sep=1pt] (B) at (1,1.5) {$b$};
\node[draw, circle, inner sep=1pt] (C) at (2,0) {$c$};
\node[draw, circle, inner sep=1pt] (D) at (3,1.5) {$d$};
\node[draw, circle, inner sep=1pt, font=\scriptsize] (F) at (-1,1.5) {$f$};
\node[draw, circle, inner sep=1pt] (E) at (1,-1.5) {$e$};

\draw[mid arrow,thick] (A) -- (B) node[midway, left] {2};
\draw[mid arrow,thick] (B) -- (C) node[midway, right] {2};
\draw[mid arrow,thick] (C) -- (A) node[midway, below] {2};
\draw[mid arrow] (F) -- (B) node[midway, above] {$3$};
\draw[mid arrow] (F) -- (A) node[midway, left] {1};
\draw[mid arrow] (D) -- (B) node[midway, above] {1};
\draw[mid arrow] (D) -- (C) node[midway, right] {3};
\draw[mid arrow] (E) -- (C) node[midway, right] {1};
\draw[mid arrow] (E) -- (A) node[midway, left] {3};

\end{tikzpicture}
\end{center}
\caption{A non-comparability temporal graph admitting a QTTO.}
\label{fig:qtto-not-tto}
\end{figure}
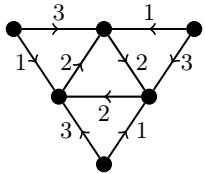

The following lemma states some simple facts about QTTO.

 \begin{lemma}\label{lem:qtto-monolabel-triangle}
 Let $O$ be a QTTO of a temporal graph $\mathcal{G}=(G,\lambda)$. Then,
\begin{itemize}
\item A directed triangle in $O$ is necessarily monolabel. 
\item $O$ is $TTO$ if only if $O$ is acyclic.
\end{itemize}
 \end{lemma}
 \begin{proof}

Suppose there exists a directed non-monolabel triangle $abc$ in $O$. Then, necessarily there exists an edge $ab$ with maximum time-label value in the triangle which is adjacent to another edge, say $bc$, such that $\lambda(ab) > \lambda(bc)$ and $\lambda(ca) \leq \lambda(ab)$. But this contradicts $O$ being QTTO. Now, for the second item, if $O$ is QTTO and acyclic then, in particular, it does not contain any directed triangle and thus is TTO. Suppose $O$ is QTTO but not acyclic and let $C$ be a minimal directed cycle. There must exist along the cycle two consecutive arcs $(a,b)$ and $(b,c)$ such that $\lambda(ab)\leq \lambda(bc)$. Since $O$ is QTTO, either $(a,c)$ or  $(c,a)$ belong to $O$. The first case would produce a shorter directed cycle, contradicting the minimality of $C$, so we can assume  $(c,a)\in O$. But then we get a directed triangle and thus $O$ is not TTO.
\end{proof}

We now define the implication graph of a temporal graph $\mathcal{G}$, that will encode the forcing rules between arcs in a QTTO described in the introduction.

\begin{definition}
The implication digraph of a temporal graph $\mathcal{G}=(G,\lambda)$, denoted \textbf{Imp($\mathcal{G}$)}, is the digraph with vertex set $AP(G)$ and with arcs from $(a,b)$ to $(c,b)$ and from $(b,c)$ to $(b,a)$ if  $\lambda(ab)\leq \lambda(bc)$ and $(ac \notin E$ or $\lambda(ac)<\lambda(bc))$.
\end{definition}

Observe that $O$ is a QTTO if and only if  for each arc from $(a,b)$ to $(c,b)$ in $Imp(\mathcal{G})$,  $(a,b)\in O$ implies $(c,b)\in O$. As we have seen, this implication digraph is not enough to capture the TTO forcing relations because of monolabel triangles. This is why we will augment this digraph using a four-vertex configuration, called \textit{correlated monolabel triangle}.

\begin{definition}
A quadruple of vertices $(a,b,c,d)$ of a temporal graph $\mathcal{G}=(G,\lambda)$ forms a \textbf{correlated monolabel triangle} if: 
\begin{itemize}
    \item $abc$ form a monolabel triangle of time-label $t$
    \item $bd \in E$ with $\lambda(bd)\leq t$
    \item $cd \in E$ with $\lambda(cd) > t$
    \item if $ad \in E$, then $\lambda(ad) < t$
\end{itemize}
\end{definition}
\begin{figure}[H]
\begin{center}

\begin{tikzpicture}[scale=0.75]
\node[draw, circle, inner sep=1pt] (D) at (0,0) {$d$};
\node[draw, circle, inner sep=1pt] (B) at (2,0) {$b$};
\node[draw, circle, inner sep=1pt] (A) at (1,-1.5) {$a$};
\node[draw, circle, inner sep=1pt] (C) at (3,-1.5) {$c$};
\draw[-,line width=0.75pt] (D) -- (B) node[midway, above] {$\leq t$};
\draw[-, dashed] (D) -- (A) node[pos=0.6, left] {$<t\,$};
\draw[-, line width=0.75pt] (C) -- (A) node[midway, below] {$t$};
\draw[-, line width=0.75pt] (C) -- (B) node[midway, right] {$t$};
\draw[-, line width=0.75pt] (A) -- (B) node[midway, left] {$t$};
\draw[line width=0.75pt] (D) .. controls (-1.8,-2) and (2.1,-3.7) .. (C) node[midway, below left] {$>t$};
\end{tikzpicture}
\vspace{-2em}
\end{center}
\caption{Correlated monolabel triangle $(a,b,c,d)$.}
\label{fig:bad-quadruple}
\end{figure}
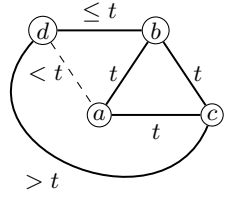

{\noindent\bf Remark:} When depicting a temporal graph, dashed labelled edges represent either a non edge or an edge satisfying the label condition.\\

For any correlated monolabel triangle $(a,b,c,d)$, it is not difficult to see that there is a directed path $(b,c)(d,c)(d,b)(a,b)$ in $Imp(\mathcal{G})$. Therefore, in any TTO, the presence of $(b,c)$  forces the presence of $(a,b)$. Since a TTO cannot contain the directed triangle $abc$, it follows that if $(b,c)$ is an arc of a TTO, then so will be $(a,c)$. Based on this observation, we propose the following definition.

\begin{definition}
The augmented implication digraph of a temporal graph $\mathcal{G}$, denoted \textbf{Aug($\mathcal{G}$)}, is the digraph obtained from $Imp(\mathcal{G})$ by adding all arcs $((b,c),(a,c))$ and $((c,a),(c,b))$, where $(a,b,c,d)$ forms a correlated monolabel triangle. We say that an arc $(x,y)$ \bfit{forces} an arc $(u,v)$ if there is a directed path from $(x,y)$ to $(u,v)$ in $Aug(\mathcal{G})$.
\end{definition}

The following lemma is a direct consequence of the definition of the arcs of $Imp(\mG)$ and $Aug(\mG)$. It reflects the law of contraposition if we think of the vertices of $Aug(\mG)$ as Boolean variables, as in the introduction.

\begin{lemma}\label{lem:contrapos}
Let  $uv$ and $xy$ be edges of a temporal graph $\mG$. If $(u,v)$ forces $(x,y)$, then $(y,x)$ forces $(v,u)$.
\end{lemma}

\begin{definition}
If $O$ is a partial orientation of a temporal graph $\mathcal{G}$, we denote by $O^+$ (resp. $O^-$) the set of vertices $(u,v)$ of $Aug(\mathcal{G})$ that are forced by (resp. that force) some arc in $O$. An orientation $O$ of $\mG$ is an Almost-TTO (or \textbf{ATTO})  if $O^+=O$.
\end{definition}

The additional restrictions captured by $Aug(\mathcal{G})$ do not follow directly from the definition of temporal transitive orientations. However, as we have seen before, if $(x,y)$ forces $(u,v)$, then any TTO of $\mathcal{G}$ that contains $(x,y)$ must also contain $(u,v)$. In other words, \bfit{every TTO is an ATTO}. Again the converse is not true (consider a monolabel directed triangle), but the main theorem of this section (Theorem \ref{theo:gouila-houri-generalization}) will provide the desired analogue of Ghouila-Houri's theorem: a temporal graph admits a $TTO$ if and only if it admits an $ATTO$.

The definition of the augmented implication digraph $Aug(\mathcal{G})$ immediately implies the following useful lemma.

\begin{lemma}
Let $\mathcal{G}$ be a temporal graph and let $O$ be an ATTO of $\mathcal{G}$. If $(a,b,c,d)$ is a correlated monolabel triangle, then triangle $abc$ cannot be directed in $O$.
\label{lem:non-dir-corr-triangle}
\end{lemma}

Since $Imp(\mG)$ is a subgraph of $Aug(\mG)$ it is clear that every ATTO is a QTTO.  The converse is not true, as one can take as an example the QTTO represented in Figure \ref{fig:qtto-not-tto}, which is not ATTO as $abc$ is a directed triangle belonging to a correlated monolabel triangle.

The following definition will play a crucial role in establishing the equivalence between admitting a TTO and an ATTO.

\begin{definition}
We will say an arc $(x,y)$ is \textbf{necessary} if $(y,x)$ forces $(x,y)$. We will denote by \textbf{Ness($\mathcal{G}$)} the set of necessary arcs of $\mathcal{G}$. 
\end{definition}

\begin{lemma}\label{lem:necessary}
Let $\mG$ be a temporal graph. Then $Ness(\mG)^+=Ness(\mG)$ and $Ness(\mG)\subset O$ for any $O$ ATTO of $\mG$. 
\end{lemma}
\begin{proof}
Consider $(x,y) \in Ness(\mathcal{G})$ and assume $(x,y)$ forces $(u,v)$. By Lemma \ref{lem:contrapos} $(v,u)$ forces $(y,x)$, and by definition of $\Ness(\mG)$ $(y,x)$ forces $(x,y)$. Hence $(v,u)$ forces $(u,v)$ and thus $Ness(\mathcal{G})^+ = Ness(\mathcal{G})$. Now if $O$ is an ATTO of $\mathcal{G}$, since an arc and its reverse cannot be both in $O$, all arcs in $Ness(\mathcal{G})$ must be contained in $O$.
\end{proof}

By Lemma \ref{lem:qtto-monolabel-triangle}, any TTO is acyclic, so if a temporal graph $\mG$ admits a TTO, then the arcs in $Ness(\mathcal{G})$ must induce an acyclic orientation in $\mG$. In fact, the key ingredient to prove that $\mG$ admits a TTO if and only if it admits an ATTO will be the observation that both conditions are equivalent to $Ness(\mG)$ containing no directed cycle. Note that the fact that $Ness(\mG)$ is acyclic for any comparability temporal graph directly  implies that the graph represented in Figure \ref{fig:qtto-not-tto} does not admit any TTO, as it is not difficult to see that the three arcs $(a,b), (b,c),(c,a)$ are all in $Ness(\mathcal{G})$.

Let us now propose two lemmas, reminiscent of a statement called the \bfit{Triangle Lemma} in \cite{Golumbic2nd}, that will be central in the proof of our main theorem.

\begin{lemma}\label{lem:triangle}Let $O$ be an ATTO of a temporal graph $\mathcal{G}=(G,\lambda)$. If $uvw$ is a directed triangle in $O$ and there is an arc $(a,b) \in O$ such that $((a,b),(u,v)) \in Aug(\mathcal{G})$, then $abw$ is also a directed triangle in $O$.
\end{lemma}

\begin{proof} Let $uvw$ be a directed triangle in $O$. Since every ATTO is a QTTO, by Lemma \ref{lem:qtto-monolabel-triangle}, $uvw$ is necessarily monolabel and we denote by $t$ be the common time-label. We will distinguish cases on whether the arc from $(a,b)$ to $(u,v)$ is an arc in $Imp(\mathcal{G})$ or not, and whether $a=u$ or $b=v$. 

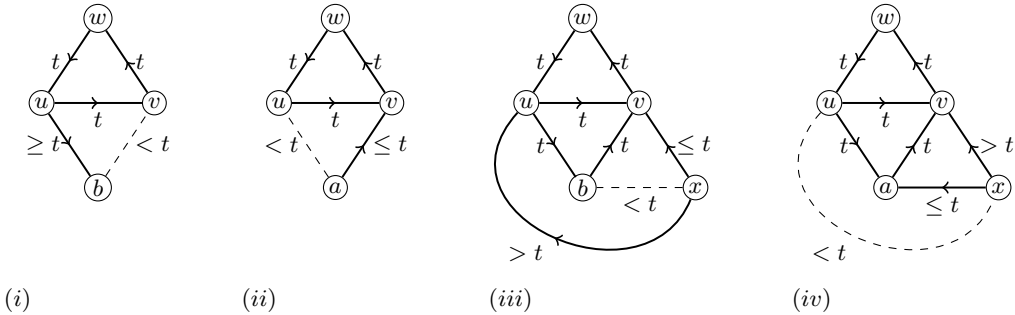
\begin{figure}[H]
\begin{center}
\begin{tikzpicture}[scale=0.75]
\node[draw, circle, inner sep=1pt] (U) at (0,0) {$u$};
\node[draw, circle, inner sep=1pt] (W) at (1,1.5) {$w$};
\node[draw, circle, inner sep=1pt] (V) at (2,0) {$v$};
\node[draw, circle, inner sep=1pt] (B) at (1,-1.5) {$b$};
\node[text width=1cm] (I) at (0,-3.5) {$(i)$};
\draw[mid arrow] (U) -- (B) node[midway, left] {$\geq t$};
\draw[mid arrow] (U) -- (V) node[midway, below] {$t$};
\draw[mid arrow] (V) -- (W) node[midway, right] {$t$};
\draw[mid arrow] (W) -- (U) node[midway,left] {$t$};
\draw[-,dashed] (B) -- (V) node[midway,right] {$<t$};
\end{tikzpicture} \hspace{3ex}
\label{fig:case1}
\begin{tikzpicture}[scale=0.75]
\node[draw, circle, inner sep=1pt] (U) at (0,0) {$u$};
\node[draw, circle, inner sep=1pt] (W) at (1,1.5) {$w$};
\node[draw, circle, inner sep=1pt] (V) at (2,0) {$v$};
\node[draw, circle, inner sep=1pt] (A) at (1,-1.5) {$a$};
\node[text width=1cm] (I) at (0,-3.5) {$(ii)$};
\draw[mid arrow] (A) -- (V) node[midway, right] {$\leq t$};
\draw[mid arrow] (U) -- (V) node[midway, below] {$t$};
\draw[mid arrow] (V) -- (W) node[midway, right] {$t$};
\draw[mid arrow] (W) -- (U) node[midway,left] {$t$};
\draw[-, dashed] (A) -- (U) node[midway,left] {$<t$};
\end{tikzpicture}
\begin{tikzpicture}[scale=0.75]
\node[draw, circle, inner sep=1pt] (U) at (0,0) {$u$};
\node[draw, circle, inner sep=1pt] (W) at (1,1.5) {$w$};
\node[draw, circle, inner sep=1pt] (V) at (2,0) {$v$};
\node[draw, circle, inner sep=1pt] (B) at (1,-1.5) {$b$};
\node[draw, circle, inner sep=1pt] (X) at (3,-1.5) {$x$};
\node[text width=1cm] (I) at (0,-3.5) {$(iii)$};
\draw[mid arrow] (U) -- (V) node[midway, below] {$t$};
\draw[mid arrow] (V) -- (W) node[midway, right] {$t$};
\draw[mid arrow] (W) -- (U) node[midway, left] {$t$};
\draw[mid arrow] (B) -- (V) node[midway, right] {$t$};
\draw[mid arrow] (U) -- (A) node[midway, left] {$t$};
\draw[mid arrow] (X) -- (V) node[midway, right] {$\leq t$};
\draw[dashed] (X) -- (B) node[midway, below] {$< t$};
\draw[mid arrow] (X) .. controls (2.1,-3.7) and  (-1.8,-2) .. (U) node[midway, below left] {$> t$};
\end{tikzpicture} 
\begin{tikzpicture}[scale=0.75]
\node[draw, circle, inner sep=1pt] (U) at (0,0) {$u$};
\node[draw, circle, inner sep=1pt] (W) at (1,1.5) {$w$};
\node[draw, circle, inner sep=1pt] (V) at (2,0) {$v$};
\node[draw, circle, inner sep=1pt] (A) at (1,-1.5) {$a$};
\node[draw, circle, inner sep=1pt] (X) at (3,-1.5) {$x$};
\node[text width=1cm] (I) at (0,-3.5) {$(iv)$};
\draw[mid arrow] (U) -- (V) node[midway, below] {$t$};
\draw[mid arrow] (V) -- (W) node[midway, right] {$t$};
\draw[mid arrow] (W) -- (U) node[midway,left] {$t$};
\draw[mid arrow] (A) -- (V) node[midway, right] {$t$};
\draw[mid arrow] (U) -- (A) node[midway,left] {$t$};
\draw[mid arrow] (X) -- (V) node[midway, right] {$>t$};
\draw[mid arrow] (X) -- (A) node[midway, below] {$\leq t$};
\draw[dashed] (U) .. controls (-1.8,-2) and (2.1,-3.7) .. (X) node[midway, below left] { $<t$};
\end{tikzpicture}
\end{center}
\end{figure}

\begin{enumerate}[label=\roman*)]
 
\item Let $a=u$ and $((u,b),(u,v))$ be an arc of $Imp(\mathcal{G})$. Then, $\lambda(ub)\geq \lambda (uv)$ and either $bv\not\in E$ or $\lambda(bv)<\lambda(ub)$. But now, $wub$ is a directed path with $\lambda(wu) \leq \lambda(ub)$. Since $O$ is ATTO, we must have $wb\in E$ with $\lambda(wb)\geq \lambda(ub)\geq t$. It is not possible to orient $wb$ from $w$ towards $b$ in $O$, as it would create the time-respecting directed path $vwb$ and by hypothesis $bv\not\in E$ or $\lambda(bv)<t$, which would contradict $O$ being ATTO. We conclude that $(b,w)\in O$, which implies triangle $ubw = abw$ is directed.

\vspace{1ex}     
     
\item Let $b=v$ and $((a,v),(u,v))$ be an arc of $Imp(\mathcal{G})$. For this second case,  $\lambda(av)\leq t$ and either $au\not\in E$ or $\lambda(au)<t$. Since $O$ is ATTO and $avw$ is a time-respecting directed path, we must have $aw\in E$ and $\lambda(aw)\geq t$. Suppose $aw$ is oriented  from $a$ towards $w$ in $O$ and consider the directed path $awu$. Since $au\not\in E$ or $\lambda(au)<t$ and, once again, $O$ is ATTO, we must have $\lambda(aw)>t$. Then, the quadruple $(u,v,w,a)$ forms a correlated monolabel triangle. But this is a contradiction by Lemma \ref{lem:non-dir-corr-triangle} since triangle $uvw$ is directed in $O$, an ATTO. Hence, $aw$ must be oriented from $w$ towards $a$ in $O$, forming the directed triangle $avw = abw$.

\vspace{1ex}

\item Let $a=u$ and $((u,b),(u,v))$ be an arc of $Aug(\mathcal{G})$ but not of $Imp(\mathcal{G})$. For this case to be produced, we claim the existence of a vertex $x$ such that $(b,v,u,x)$ is a correlated monolabel triangle: as $(v,u)$ forced $(b,v)$ in $Imp(G)$, arcs $((v,u),(b,u))$ and $((u,b),(u,v))$ were added to $Aug(G)$. Suppose $(v,x) \in O$. Since $O$ is ATTO and $vxu$ form a directed path with $\lambda(xu) > \lambda(vu)$, this forces $(u,x) \in O$ which in turn forces $(w,x) \in O$ with $\lambda(wx) \geq \lambda(ux) > t$. But then, $vwx$ forms a time-respecting path with $\lambda(wx) > t$ and $\lambda(vx) \leq t$, a contradiction to $O$ being ATTO. Hence, $(x,v) \in O$. Then, necessarily $(b,v)\in O$ and $wx \in E$ with $\lambda(wx) \geq t$. Suppose $(w,x) \in O$, then $vwx$ must form a monolabel triangle. That is, $\lambda(xv) = \lambda(wx) = t$. But then, both possible orientations $(u,x)$ and $(x,u)$ of $xu$ lead to a contradiction as the triplet $uwx$ would fail to satisfy the ATTO conditions. Therefore, we claim $(x,w) \in O$. Also, we deduce $(x,u) \in O$ as otherwise $xuw$ would form a non-monolabel directed triangle. As $wub$ form a monolabel directed path, either $(w,b)$ or $(b,w)$ belong to $O$, with $\lambda(bw) \geq t$. If $(w,b) \in O$, we obtain a correlated monolabel triangle $(b,v,w,x)$ directed in an ATTO, a contradiction by Lemma \ref{lem:non-dir-corr-triangle}. Finally, if $(b,w) \in O$ we have found our directed triangle $ubw = abw$.

\vspace{1ex}

\item Let $b=v$ and $((a,v),(u,v))$ be an arc of $Aug(\mathcal{G})$ but not of $Imp(\mathcal{G})$. Similarly to the previous case, there must exist a vertex $x$ such that $(u,a,v,x)$ is a correlated monolabel triangle. As observed when such quadruples were defined, since $(a,v)$ belongs to $O$, then so do $(x,v)$, $(x,a)$ and $(u,a)$. Consider the monolabel directed path $avw$. As $O$ is ATTO, $aw \in E$ with $\lambda(aw) \geq t$. Notice that if $(a,w) \in O$, then $uaw$ is a directed triangle and if $(w,a) \in O$, then triangle $wav$ is directed. As $O$ is ATTO, the only way for it to contain a directed triangle is for it to be monolabel, we deduce $\lambda(aw) =t$. Assume for contradiction that $aw$ is oriented from $a$ towards $w$. Then, the directed path $xaw$, with $\lambda(xa) \leq \lambda(aw)$, implies the existence of an edge $xw$ such that $\lambda(xw) \geq t$. It cannot be oriented from $w$ towards $x$ because it would create a non-monolabel directed triangle $wxv$. Therefore, $(x,w)\in O$. Since by hypothesis $ux \notin E$ or $\lambda(ux) < t$, if $\lambda(xw) = t$, $O$ would fail being ATTO. Hence, $\lambda(xw) > t$. But now, $(u,a,w,x)$ form a correlated monolabel triangle directed in $O$, a contradiction by Lemma \ref{lem:non-dir-corr-triangle}.
\end{enumerate}
\end{proof}

\begin{lemma}
Let $O$ be an ATTO of a temporal graph   $\mathcal{G}$. If $uvw$ is a directed triangle in $O$ and if there exists an arc $(a,b) \in O$ forcing $(u,v)$ in $Aug(\mathcal{G})$, then $abw$ is also a directed triangle in $O$.
\label{lem:iterative-triangle}
\end{lemma}

\begin{proof}
As $(a,b)$ forces $(u,v)$, there exists a directed path from $(a,b)$ to $(u,v)$ in $Aug(\mathcal{G})$. Then, we can repeatedly apply Lemma \ref{lem:triangle} starting from the second to last arc in the directed path, i.e., arc directly forcing $(u,v)$, until reaching arc $(a,b)$. At each step, we get a new  directed triangle containing an arc forced by the immediately previous arc in the path, this allows us to continue to apply Lemma \ref{lem:triangle}.  When reaching arc $(a,b)$, Lemma \ref{lem:triangle} will finally guarantee that $abw$ is also a directed triangle in $O$.
\end{proof}

\begin{lemma}
\label{lem:atto-directed-triangle-no-arc-in-ness}
Let $O$ be an ATTO of a temporal graph $\mathcal{G}$. No directed triangle in $O$ can have an arc in $Ness(\mathcal{G})$.
\end{lemma}
\begin{proof}
Assume by contradiction that  $uvw$ is a directed triangle in $O$ such that $(u,v)\in Ness(\mathcal{G})$. Then, as $(v,u)$ forces $(u,v)$, there exists a directed path $\mu$ in $Aug(\mathcal{G})$ from $(v,u)$ to $(u,v)$. By Lemma \ref{lem:contrapos}, $Aug(\mathcal{G})$ will also contain the \enquote{contraposite} path of $\mu$, that is the path $\mu'$ from $(v,u)$ to $(u,v)$ consisting of all arcs $((y_2,x_2),(y_1,x_1))$ such that  $((x_1,y_1),(x_2,y_2))$ is an arc in $\mu$. Let $(x,y)$ be the first node in $\mu$ such that $(x,y) \in O$. Since $O$ is ATTO, all nodes in the subpath of $\mu$ from $(x,y)$ to $(u,v)$ will also be included in $O$. Consider $(x',y')$ the node immediately before $(x,y)$ in $\mu$. As $(x',y') \notin O$, necessarily $(y',x') \in O$. Also, $(y',x')$ belongs to $\mu'$, and therefore forces $(u,v)$ in $Aug(\mathcal{G})$. Then, both $(y',x')$ and $(x,y)$ are arcs in $O$ forcing arc $(u,v)$ in $Aug(\mathcal{G})$ and we can apply Lemma \ref{lem:iterative-triangle} twice to state that $y'x'w$ and $xyw$ are two directed triangles in $O$. But now, since $((x',y'),(x,y)) \in Aug(\mathcal{G})$, either $x' = x$ or $y' = y$. If $x' = x$, the directed triangle $y'xw$ implies $(x,w)\in O$ while the directed triangle $xyw$ supposes the inverse orientation, i.e., $(w,x)$. If $y' = y$, we get a similar conflict with the orientation of edge $wy$. We conclude that triangles $y'x'w$ and $xyw$ cannot be both directed in $O$ and reach a contradiction.
\end{proof}

In the lemma below, we generalize to ATTO a property previously established for TTO: no directed cycle in an ATTO of a temporal graph $\mG$ is formed exclusively by arcs in $Ness(\mG)$.
\begin{lemma}
\label{lem:ness-acyclic}
If a temporal graph $\mathcal{G}=(G,\lambda)$ admits an ATTO, then $Ness(\mathcal{G})$ defines an acyclic partial orientation of $\mG$.
\end{lemma}
\begin{proof}
 Assume by contradiction that $Ness(\mathcal{G})$ induces a directed cycle in $O$ and let $C$ be a minimal such cycle. Let $O$ be an ATTO of $\mathcal{G}$ and recall that, by Lemma \ref{lem:necessary},  $Ness(\mG)\subset O$. Since $O$ is ATTO, by Lemma \ref{lem:atto-directed-triangle-no-arc-in-ness}, we get $|C| > 3$. We prove first that $C$ is monolabel, assume for contradiction this is not the case. Then, there must exist two consecutive arcs $(a,b)$ and $(b, c)$ in $C$ such that $\lambda(ab)<\lambda(bc)$. As $O$ is ATTO,  $(a,c) \in O$ with $\lambda(ac) \geq \lambda(bc) > \lambda(ab)$ and thus $(b,c)$ forces arc $(a,c)$ in $Aug(\mathcal{G})$. Since $(b,c) \in Ness(\mathcal{G})$, by Lemma \ref{lem:necessary}, $(a,c) \in Ness(\mathcal{G})$. Hence, we reached a smaller directed cycle induced by $Ness(\mathcal{G})$, which contradicts $C$ being the minimal one. $C$ is therefore monolabel and we denote by $v_1, ..., v_k$ its vertices. As $O$ is ATTO and $v_1v_2v_3$ form a monolabel directed path in $O$, then necessarily $v_1v_3 \in E$. If $(v_3,v_1) \in O$, we have found a directed triangle $v_1v_2v_3$ such that $(v_2,v_3) \in C$, a contradiction to Lemma  \ref{lem:atto-directed-triangle-no-arc-in-ness}. Then, $(v_1,v_3) \in O$. If $\lambda(v_1v_3) > t =\lambda(v_2v_3) =  \lambda(v_1v_2)$, then $(v_2,v_3)$ forces $(v_1,v_3)$ in $Aug(\mG)$ and thus, by Lemma \ref{lem:necessary}, $(v_1,v_3)\in Ness(\mG)$, yielding a directed cycle formed by arcs in $Ness(\mG)$ shorter than $C$, a contradiction. Hence $\lambda(v_1v_3) = t=\lambda(v_3,v_4)$ and we can repeat the same argument with vertices $v_1,v_3$ and $v_4$ to get $(v_1,v_4) \in O$ with $\lambda(v_1v_4)=t$. Hence, in order to avoid creating a directed triangle $v_1v_iv_{i+1}$ with $(v_i,v_{i+1}) \in C$, necessarily $(v_1,v_i)\in O$ for all $i$ such that $1 < i < k$. But then, $v_{k-1}v_kv_1$ form a directed triangle with $(v_{k-1},v_k) \in C\subset Ness(\mG)$, contradicting again Lemma \ref{lem:atto-directed-triangle-no-arc-in-ness}.
\end{proof}

We are now ready to prove our main result.

\begin{theorem}
 Let $\mathcal{G}$ be a temporal graph. The following assertions are equivalent:
\begin{enumerate}[label=\roman*)]
     \item $\mG$ admits a TTO.
     \item $\mathcal{G}$ admits an ATTO.
     \item $Ness(\mathcal{G})$ induces an acyclic digraph.
\end{enumerate}
\label{theo:gouila-houri-generalization}
\end{theorem}

\begin{proof}
The implication from $i)$ to $ii)$ is straightforward as all TTO are ATTO and the one from $ii)$ to $iii)$ is proved by Lemma \ref{lem:ness-acyclic}. To prove the implication from $iii)$ to $i)$, we will mimic the proof provided by Hell et al. in \cite{hell1995lexicographic} for the static case, by using an analogous lexicographic argument. The proof will provide a simple algorithm to decide whether a given temporal graph $\mathcal{G}$ is a comparability temporal graph and construct a TTO of $\mathcal{G}$ if the answer is YES, as detailed in Subsection \ref{subsect:algorithmic-aspects}. Let $O$ be a partial ATTO of $\mathcal{G}$ containing $Ness(\mathcal{G})$. We start by showing that taking any edge $ab$ not yet oriented in $O$, ensures that $O'=O\cup\{(a,b)\}^+$ will also be an ATTO of $\mathcal{G}$. As $O$ is ATTO, it is clear that $O'^+=O'$ holds. Assume for contradiction that $O'\cap O'^{R} \neq \emptyset$, i.e., there exists a pair of vertices $u,v$ such that $\{(u,v),(v,u)\}\subset O'$. Since $O$ is ATTO, both arcs cannot belong to $O$, so assume $(u,v)\in \{(a,b)\}^+$. It is not possible that $(v,u)\in \{(a,b)\}^+$ as well, as this would imply $(b,a)\in \{(u,v)\}^+$ and thus $(b,a)\in \{(a,b)\}^+$, meaning $(b,a) \in Ness(G)$ which contradicts the assumption of edge $ab$ being not yet oriented in $O$. Hence, $(v,u)\in O$. Observe that $(u,v)\in \{(a,b)\}^+$ implies $(b,a)\in \{(v,u)\}^+$. But now, since $O^+=O$, we must have $(b,a)\in O$, again contradicting the assumption of $ab$ unoriented in $O$. As $Ness(\mathcal{G})$ induces an acyclic digraph, there exists an order $\sigma$ on the vertices of $\mathcal{G}$ such that all arcs in $Ness(\mathcal{G})$ are oriented forward, i.e., if $(u,v)\in Ness(\mathcal{G})$ then $u\prec_{\sigma} v$. We propose to construct an orientation $O$ in a greedy lexicographic fashion, following the order defined by $\sigma$. We initially set $O=Ness(\mathcal{G})$. Then, while $O$ is not yet a complete orientation of $\mathcal{G}$, we repeat what follows. We consider the unoriented edge $ab$, with $a \prec_{\sigma} b$, such that it is lexicographically minimal with respect to $\sigma$, and we add $\{(a,b)\}^+$ to $O$. By the observation above, as $O$ was initially an ATTO of $\mathcal{G}$ containing $Ness(\mathcal{G})$, the final orientation obtained by this procedure will also be ATTO. It remains to show that $O$ does not contain directed triangles. Assume by contradiction it contains at least one, and consider the monolabel directed triangle $uvw$ that is lexicographically minimal with respect to the position in $\sigma$ of its three vertices. Then, $uvw$ contains at least one arc oriented backwards by $\sigma$, without loss of generality take $(u,v)$ as this backward arc with $v \prec_{\sigma} u$. Since it was oriented backwards, $(u,v)\notin Ness(\mathcal{G})$ and it is not possible that $(u,v)$ was chosen as next unoriented edge $ab$ in the algorithm above. Hence, the orientation of $(u,v)$ was assigned because $(u,v) \in \{(a,b)\}^+$ for some edge $ab$, chosen as minimal unoriented edge at some step in our greedy algorithm. As at this step $uv$ was also an unoriented edge, this implies that $ab$ is lexicographically smaller that $uv$. But now, since there is a directed path from $(a,b)$ to $(u,v)$ in $Aug(\mathcal{G})$, by applying Lemma \ref{lem:iterative-triangle}, we get that $abw$ is also a directed triangle in $O$, contradicting the lexicographic minimality of $uvw$.  \end{proof}

Note that not all ATTO are TTO, as an ATTO might contain directed monolabel triangles, that do not form any correlated monolabel triangle with a fourth vertex. Take a single monolabel triangle as an example where a directed orientation of its arcs is an ATTO. Exactly like in the static case, where it is important to note that the existence of a quasi-transitive orientation $O$ of $G$ proves the existence of a transitive orientation $O'$ of $G$, although $O$ might not be transitive as it can contain directed triangles. This is why the lexicographic strategy is the key argument in our method, as it was in \cite{hell2014ordering} for the static case. 
Additionally, we ensure that Theorem \ref{theo:gouila-houri-generalization} is a correct extension of the static case. Note that static graphs can be considered as monolabel temporal graphs. As no correlated monolabel triangle can be induced by a monolabel temporal graph $\mathcal{G}$, it follows that in that case $Aug(\mathcal{G})$ is equal to $Imp(\mathcal{G})$. Then, it is easy to see that the double implication between $i)$ and $ii)$ corresponds to the Ghouila-Houri result, i.e. Theorem \ref{theo:GH}. Moreover, in the static scenario the forcing relations are symmetric, i.e. if $(a,b)$ forces $(b,a)$ in $Aug(\mathcal{G})$, then $(b,a)$ forces $(a,b)$ in return. Hence,  if $(a,b) \in Ness(\mathcal{G})$ then $(b,a)$ also belongs to $ Ness(\mathcal{G})$. Therefore, $Ness(\mathcal{G}) \neq \emptyset$ implies that $\mathcal{G}$ is not a comparability graph as $Ness(\mathcal{G})$ will automatically induce a cycle as every necessary arc will be contained in a 2-cycle.

\subsection{Algorithmic aspects}
\label{subsect:algorithmic-aspects}

The proof presented for Theorem \ref{theo:gouila-houri-generalization} establishes the algorithm we will propose for deciding whether a given temporal graph $\mathcal{G}$ is a comparability temporal graph, and if so, obtaining a TTO of its edges. The strategy will start by obtaining the set $Ness(\mathcal{G})$ of necessary arcs of $\mathcal{G}$. For this, it is first required to construct the implication digraph $Imp(\mathcal{G})$ and then to add the appropriate constraints to obtain the augmented digraph $Aug(\mathcal{G})$. We start by observing that the size of both digraphs is in $O(nm)$, which is the order of the maximum number of triplets of vertices, forming a triangle or a three-vertex path and therefore potentially producing an implication constraint. Consider $f(n,m)$ and $g(n,m)$ the complexity functions of computing digraphs $Imp(\mathcal{G})$ and $Aug(\mathcal{G})$ respectively. Then, by the previous observations, we claim that deciding whether $\mathcal{G}$ admits a QTTO (resp. ATTO) can be solved in $\max\{f(n,m), O(nm)\}$ (resp. $\max\{g(n,m), O(nm)\}$) by a 2-SAT linear time algorithm. To obtain $Imp(\mathcal{G})$, only triangles and three-vertex paths need to be evaluated as no other structure can produce an implication, i.e., an arc in the digraph. Then, it is easy to see that $f(n,m) \in O(nm)$. However, for $Aug(\mathcal{G})$, we need to additionally consider correlated monolabel triangles, formed by a specific quadruple of vertices. To do this, two different strategies can be implemented, and the structure of the graph will determine which of the two is most efficient. First, one can consider for each monolabel triangle $abc$, all other vertex $d$ and verify whether the quadruple forms a correlated monolabel triangle. If we consider $k$ as the number of monolabel triangles, then clearly this procedure is in $O(kn)$. Alternatively, one can consider every pair of edges $ac, bd$ in $E$ and evaluate whether $(a,b,c,d)$ satisfies the desired configuration. This can be done in $O(m^2)$. When $k$ is small, the first strategy will be more efficient, however this is not always the case as $k \in O(nm)$. Therefore, we state that $g(n,m) \in \min\{O(kn),O(m^2)\}$, and as Theorem \ref{theo:gouila-houri-generalization} proves that any temporal graph $\mathcal{G}$ admitting an ATTO is a comparability temporal graph, we claim that the recognition problem for comparability temporal graphs can be solved in $O(nm + \min\{kn,m^2\})$. Once the digraph $Aug(\mathcal{G})$ is constructed, it is straightforward to obtain $Ness(\mathcal{G})$ by Tarjan's \cite{Tarjan72} linear strategy in the size of the digraph, that is $O(nm)$. Again by Theorem \ref{theo:gouila-houri-generalization}, we know that $Ness(\mathcal{G})$ is a directed acyclic graph and we can obtain in $O(nm)$ an order $\sigma$ such that all necessary arcs are oriented from left to right according to $\sigma$, i.e., all $(a,b) \in Ness(\mathcal{G})$ satisfy $a \prec_{\sigma} b$. We can extend $\sigma$ with the rest of the vertices in $\mathcal{G}$ in an arbitrary way. Then, we consider $Ness(\mathcal{G})$ as the initial orientation $O$ and we greedily orient the remaining edges by selecting at every step the unoriented edge $ab$, with $a \prec_{\sigma} b$, such that $a$ is lexicographically minimum. We then fix orientations $\{(a,b)\}^{+}$ in $O$. Using the same arguments as in proof of Theorem $\ref{theo:gouila-houri-generalization}$, we state that this results in a TTO of $\mathcal{G}$. Note that adding $(a,b)$ to $O$ will fix the orientation of all the arcs forced in $\{(a,b)\}^{+}$. Then, after orienting $\{(a,b)\}^{+}$ we can remove all these vertices from $Aug(\mathcal{G})$, as all vertices in $\{(b,a)\}^{-}$, and claim the complexity of the greedy orientation construction can be done in $O(nm)$. To conclude, we state that the overall time complexity of the algorithm is of $O(nm + \min\{kn,m^2\})$.

\begin{theorem}
Deciding if $\mathcal{G}$ is a comparability temporal graph and constructing a TTO if the answer is YES can be done in $O(nm + \min\{kn,m^2\})$.
\end{theorem}

The bottleneck of our algorithm can be the computation of the correlated monolabel triangles when $k \in O(mn)$ as the construction of the digraph $Aug(\mathcal{G})$ will require $O(m^2)$. On the other hand, since up to our knowledge there is not yet a linear time algorithm to verify whether an orientation is transitive in the static case, and temporal transitive orientations generalize transitive ones, this supposes a lower bound for the comparability temporal graph recognition problem.

\subsection{Multilabel temporal transitive orientations}

Since most temporal networks allow resources to traverse their arcs at multiple given times, it is natural to try to formulate an extension to the notion of temporal transitivity to multilabel temporal graphs, i.e., temporal graphs $\mathcal{G}=(G,\lambda)$ with $G=(V,E)$ and $\lambda: E \rightarrow 2^{\mathds{N}}$. We shortly introduce the motivation and main definitions, further clarifications can be found in the Appendix section of the related version of this article. Let us recall the initial motivation behind the model proposed by Mertzios et al.\cite{mertzios2025complexity}. The authors introduced temporal transitive orientations as a way to avoid \enquote{shortcuts} in the transmission of information by nodes in a communication network. This way, if node $a$ sends information to node $b$ and node $b$ then sends it to node $c$, the latter should be able to verify this information from source $a$ at a later, or equal, time than when receiving the message via $b$. In the multilabel setting, it is considered that new information can reach node $a$ and be transmitted to $b$ multiple times, at the moments determined by the time-label set associated with the arc connecting both nodes. Therefore, the arc connecting node $a$ with node $c$ should not allow the information to be transmitted at an earlier time to the one of $b$ sending the information to $c$. We formalized this idea as follows. First, we introduce the operator that will allow us to compare two time-label sets in the multilabel scenario.

\begin{definition}
Let $S_1,S_2$ be two sets in $2^{\mathds{N}}$, $S_1 \preceq S_2$ (resp. $S_1 \prec S_2$) if there exist $x \in S_1$ and $y  \in S_2$ such that $x \leq y$ (resp. $x < y$).
\end{definition}

Let us denote by  $\hspace{0.5ex} \cancel{\preceq} \hspace{0.5ex}$ (resp. $\cancel{\prec}$) the negation. Notice that $S_1 \hspace{0.5ex} \cancel{\preceq} \hspace{0.5ex} S_2$ (resp. $S_1 \hspace{0.5ex} \cancel{\prec} \hspace{0.5ex} S_2$) implies that all elements in $S_1$ are strictly greater (resp. greater or equal) than all elements in $S_2$.

\begin{definition}
Let $\mathcal{G}=(G,\lambda)$ be a multilabel temporal graph and $O$ an orientation of $\mathcal{G}$. We will say $O$ is a \bfit{multilabel temporal transitive orientation} (MTTO) if whenever $(a,b),(b,c) \in O$ with $\lambda(ab) \preceq \lambda(bc)$, then necessarily $(a,c) \in O$ with $\lambda(ac) \hspace{0.5ex} \cancel{\prec} \hspace{0.5ex} \lambda(bc)$.
\end{definition}

In the multilabel scenario, the equivalent to monolabel triangles considered in the previous subsections will be monolabel triangles with singleton time-label sets. Then, we can define \bfit{almost-multilabel temporal transitive orientations} (AMTTO) as the analogous to the simple temporal case. This leads to an equivalent characterization of comparability multilabel temporal graphs and lexicographic strategy for constructing a MTTO of a YES instance.

\begin{theorem}
 Let $\mathcal{G}$ be a multilabel temporal graph. Then, the following assertions are equivalent:
\begin{enumerate}[label=\roman*)]
     \item $\mathcal{G}$ admits a MTTO
     \item $\mathcal{G}$ admits an AMTTO
     \item $Ness(\mathcal{G})$ induces an acyclic digraph
\end{enumerate}

Furthermore, there is a $O(nm + \min\{kn,m^2\})$-time algorithm to decide if $\mathcal{G}$ is a comparability multilabel temporal graph and construct a MTTO if the answer is YES, with $k$ being the number of monolabel triangles with a singleton time-label set associated to their edges.
\label{theo:mtto-gouila-houri-generalization}
\end{theorem}


\section{Characterizations by forbidden temporal ordered patterns}
\label{sect:patterns}

All proofs of results in this section are included in the Appendix.

\vspace{1ex}
\textbf{Temporal ordered patterns.} In graph theory, a property is  \bfit{hereditary} if it is stable under the induced subgraph relation. As such, hereditary properties can be characterized by a, possibly infinite, set of forbidden induced subgraphs. For some hereditary properties, interesting insights can be obtained by proving the existence of orderings of the graph's vertices avoiding induced ordered substructures, usually called patterns. In this section we extend these considerations to the temporal setting and prove such characterizations for the classes related to TTO. An \bfit{ordered temporal graph} is a temporal graph equipped with an ordering of its vertices. Two ordered temporal graphs are \bfit{isomorphic} if there is an isomorphism between them that preserves the vertex orderings. We say that an ordered temporal graph \bfit{contains} another one if by removing vertices from the first graph, one obtains an ordered temporal graph isomorphic to the second graph. A \bfit{temporal pattern} is a compact way of representing a family of ordered temporal graphs, given as an ordered temporal \bfit{trigraph}: between every pair of vertices there is either a non edge, a full edge or a dashed edge. In the latter two cases, the time-label is given as a variable, and the pattern will be defined using some inequalities between those variables. An example is shown on Figure \ref{fig:orderedpatterns-tto}. A \bfit{realization} of a temporal pattern, is an ordered temporal graph obtained from the pattern by choosing for each dashed edge either a non edge or an edge, and by assigning values to the edge variables that satisfy the contraints given by the pattern. Given a collection of patterns $\mathcal{F}$, we define the class $\mathcal{C_{\mathcal{F}}}$ as the set of connected temporal graphs admitting an ordering of their vertices that do not contain (in the ordered temporal graph sense) any realization of a pattern in $\mathcal F$. We say that the ordering of the temporal graph  \bfit{avoids} all patterns in $\mathcal F$. 

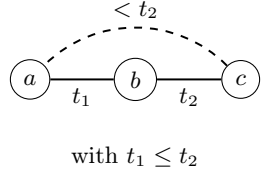
\begin{figure}[H]
\begin{center}
\begin{tikzpicture}[scale=0.55]
    \node (a) at (0, 0) [circle, draw, inner sep=2pt, fill=black] {};
    \node (b) at (2, 0) [circle, draw, inner sep=2pt, fill=black] {};
    \node (c) at (4, 0) [circle, draw, inner sep=2pt, fill=black] {};
    \draw[-, line width=0.75pt] (a) -- (b) node[midway, below] {$t_1$};
    \draw[-, line width=0.75pt] (b) -- (c) node[midway, below] {$t_2$};
    \draw[-, line width=0.75pt, bend right=45, dashed] (c) to node[midway, above] {$t_3$} (a);
\end{tikzpicture}

\vspace{1ex}
with $t_{1} \leq t_{2}$ and $t_3< t_{2}$
\end{center}
\caption{$TTO$ temporal ordered pattern.}
\label{fig:orderedpatterns-tto}
\end{figure}

Consider the temporal ordered pattern $TTO$ from Figure \ref{fig:orderedpatterns-tto} and the set of temporal ordered patterns $STTO$ illustrated in Figure \ref{fig:strict-tto-patterns} and let us propose the following characterizations. Observe that when a temporal graph is monolabel, the pattern $TTO$ corresponds to the usual pattern for characterizing the class of comparability graphs, see \cite{FeuilloleyH21} for a survey on three-vertex patterns.

\begin{theorem}\label{theo:tto-pattern-characterization}
    A temporal graph $\mathcal{G}$ admits a temporal transitive orientation if and only if it
belongs to $\mathcal{C}_{TTO}$.
\end{theorem}

Observe that when a temporal graph is monolabel, the pattern corresponds to the usual pattern for characterizing the class of comparability graphs, see \cite{FeuilloleyH21} for a survey of three-vertex patterns. This proves that the characterization by forbidden temporal ordered patterns of the comparability temporal graphs works as an extension of the static case.

\begin{figure}[ht]
\begin{center}
\begin{tikzpicture}[scale=0.55]
    \node (a) at (0, 0) [circle, draw, inner sep=2pt, fill=black] {};
    \node (b) at (2, 0) [circle, draw, inner sep=2pt, fill=black] {};
    \node (c) at (4, 0) [circle, draw, inner sep=2pt, fill=black] {};
    \node (d) at (6, 0) [circle, draw, inner sep=2pt, fill=black] {};
    \draw[-, line width=0.75pt] (a) -- (b) node[midway, below] {$t_1$};
    \draw[-, line width=0.75pt, bend left=45] (b) to node[midway, below right] {$t_2$} (d);
    \draw[-, line width=0.75pt, bend left=45, dashed] (a) to node[midway, above] {$< t_2$} (d);
    \draw[-, line width=0.75pt] (b) to node[midway, below] {$t_3$} (c);
    \node[text width=3cm] (I) at (3.5,-2.5) {(1) with $t_{1} < t_{2}$ and $t_2 \neq t_3$};
\end{tikzpicture} \hspace{4ex}
\begin{tikzpicture}[scale=0.55]
    \node (a) at (0, 0) [circle, draw, inner sep=2pt, fill=black] {};
    \node (b) at (2, 0) [circle, draw, inner sep=2pt, fill=black] {};
    \node (c) at (4, 0) [circle, draw, inner sep=2pt, fill=black] {};
    \node (d) at (6, 0) [circle, draw, inner sep=2pt, fill=black] {};
    \draw[-, line width=0.75pt] (a) -- (b) node[midway, below] {$t_1$};
    \draw[-, line width=0.75pt, bend left=45] (b) to node[midway, below right] {$t_2$} (d);
    \draw[-, line width=0.75pt,bend left=45, dashed] (a) to node[midway, above] {$< t_2$} (d);
    \draw[-, line width=0.75pt] (c) to node[midway, below] {$t_3$} (d);
    \node[text width=3cm] (I) at (3.5,-2.5) {(2) with $t_{1} < t_{2}$ and $t_2 \neq t_3$};
\end{tikzpicture} \hspace{4ex}
\begin{tikzpicture}[scale=0.55]
    \node (a) at (0, 0) [circle, draw, inner sep=2pt, fill=black] {};
    \node (b) at (2, 0) [circle, draw, inner sep=2pt, fill=black] {};
    \node (c) at (4, 0) [circle, draw, inner sep=2pt, fill=black] {};
    \node (d) at (6, 0) [circle, draw, inner sep=2pt, fill=black] {};
    \draw[-, line width=0.75pt] (a) -- (b) node[midway, below] {$t_1$};
    \draw[-, line width=0.75pt] (b) -- (c) node[midway, below] {$t_2$};
    \draw[-, line width=0.75pt,dashed] (c) to node[midway, below] {$< t_3$} (d);
    \draw[-, line width=0.75pt,bend right=45, dashed] (c) to node[midway, above] {$< t_2$} (a);
    \draw[-, line width=0.75pt, bend left=45] (b) to node[midway, above] {$t_3$} (d);
    \node[text width=3cm] (I) at (3.5,-2.5) {(3) with $t_{1} < t_{3}$ and $t_2 < t_3$};
\end{tikzpicture}
\vspace{2ex}
\end{center}

\begin{center}
\begin{tikzpicture}[scale=0.55]
    \node (a) at (0, 0) [circle, draw, inner sep=2pt, fill=black] {};
    \node (b) at (2, 0) [circle, draw, inner sep=2pt, fill=black] {};
    \node (c) at (4, 0) [circle, draw, inner sep=2pt, fill=black] {};
    \node (d) at (0, -1.5) [circle, draw, inner sep=2pt, fill=black] {};
    \draw[-, line width=0.75pt] (b) -- (c) node[midway, below] {$t_3$};
    \draw[-, line width=0.75pt, dashed] (a) to node[midway, above] {$< t_3$} (b);
    \draw[-, line width=0.75pt, bend right=65] (c) to node[midway, above] {$t_2$} (a);
    \draw[-, line width=0.75pt, bend right=45] (d) to node[midway, above left] {$t_1$} (b);
    \draw[-, line width=0.75pt, bend right=45, dashed] (d) to node[midway, below] {$< t_3$} (c);
    \node[text width=5cm] (I) at (3,-3.5) {(4) with $t_{1} < t_{2}$ and $t_{2} < t_{3}$};
\end{tikzpicture} \hspace{4ex}
\begin{tikzpicture}[scale=0.55]
    \node (a) at (0, 0) [circle, draw, inner sep=2pt, fill=black] {};
    \node (b) at (2, 0) [circle, draw, inner sep=2pt, fill=black] {};
    \node (c) at (4, 0) [circle, draw, inner sep=2pt, fill=black] {};
    \node (d) at (6, 0) [circle, draw, inner sep=2pt, fill=black] {};
    \node (e) at (0, -1.5) [circle, draw, inner sep=2pt, fill=black] {};
    \draw[-, line width=0.75pt] (a) -- (b) node[midway, below] {$t_2$};
    \draw[-, line width=0.75pt] (b) -- (c) node[midway, below] {$t_4$};
    \draw[-, line width=0.75pt] (c) -- (d) node[midway, below] {$t_3$};
    \draw[-, line width=0.75pt,bend right=45, dashed] (c) to node[midway, above] {$< t_4$} (a);
    \draw[-, line width=0.75pt, bend left=45] (b) to node[midway, above] {$t_5$} (d);
    \draw[-, line width=0.75pt, bend right=45] (e) to node[midway, above] {$t_1$} (c);
    \draw[-, line width=0.75pt,bend right=45, dashed] (e) to node[midway, below] {$ < t_3$} (d);
    \node[text width=5cm] (I) at (3.5,-4.5) {(5) with $t_{1} < t_{3}$, $t_{2} < t_{4}$ and $t_3,t_4,t_5$ not all equal};
\end{tikzpicture}
\end{center}

\caption{\bfit{STTO} temporal ordered patterns.}
\label{fig:strict-tto-patterns}
\end{figure}

\begin{theorem}
    A temporal graph $\mathcal{G}$ admits a strict temporal transitive orientation if and only if it belongs to $\mathcal{C}_{STTO}$.
    \label{theo:strict-tto-pattern-characterization}
\end{theorem}

Recall that deciding whether a temporal graph is a comparability temporal graph can be solved in polynomial time whereas the recognition problem for graphs admitting a Strict TTO was proved NP-hard. We can see how this difference in the complexity of both problems also translates into a complexity gap when proposing characterizations by temporal ordered patterns. 

Analogous characterizations can be formulated for the Strong TTO and Strong Strict TTO versions.

\section{Perspectives}
\label{sect:conclusion}

In section \ref{sect:tto-algorithm}, to prove our main result, Theorem \ref{theo:gouila-houri-generalization}, we used the lexicographic argument proposed in \cite{hell1995lexicographic}. The question of whether it is possible to directly prove that any graph admitting an ATTO is a comparability temporal graph by using a similar approach to Ghouila-Houri's original proof, i.e. by module contraction and re-orientation of an ATTO, remains unresolved. For this, the concept of module should be extended to the temporal setting. Then, a set of vertices $M$ having the same neighbors outside of $M$ and connected to them by exactly the same time-labels could be considered as a temporal module. What is more, in the static case all possible transitive orientations of a graph can be described in a compact way by the modular decomposition tree, see \cite{Golumbic2nd}. In the temporal setting, recall that our algorithm  constructs a temporal transitive orientation of a temporal graph $\mathcal{G}$ by computing $Ness(\mathcal{G})$, choosing any linear extension that puts all necessary arcs forward and greedily completing the orientation in a lexicographic way. We observe that any TTO $O$ of $\mathcal{G}$ can in fact be obtained by this algorithm by choosing the appropriate linear extension of $Ness(\mathcal{G})$: it suffices to start with the linear extension that puts all arcs of $O$ forward. However, it remains unanswered whether a compact representation of all TTO can be proposed, as for the static case. It is in this context that we raise the following open problem.

\begin{PB}
\label{op:tto-temporal-modular-decomposition}
For a given comparability temporal graph $\mathcal{G}=(G,\lambda)$, is it possible to obtain a TTO by a strategy based on temporal modular decomposition? Can this decomposition provide a compact structure to describe all possible TTO of $\mathcal{G}$?
\end{PB}

In Section \ref{subsect:preliminar-properties} we showed that every static graph $G$ admits a transitive temporalization, i.e. the assignation of time-labels to its edges, resulting in a comparability temporal graph. This observation leads to the following interesting relabelling  problem.

\begin{PB}
\label{op:tto-editing-problem}
For a given temporal graph $\mathcal{G}$, how to compute the minimum number of edges for which the time-label needs to be modified in order to obtain a comparability temporal graph?
\end{PB}

Observe that the problem of assigning time-labels to a directed unlabelled graph in order to obtain a comparability temporal graph, or conclude it is not possible, can be easily solved. It is only necessary to verify the directed graph is acyclic as then one can always assign time-labels in a way that all directed paths are strictly decreasing, as shown in Subsection \ref{subsect:preliminar-properties}. Finally, can we modify the model to also consider \textit{delays}, i.e. the time it takes for a resource to traverse an edge, or to study \textit{$\Delta$-restless paths}, i.e. paths where the waiting time at a node cannot exceed a given value $\Delta$? This type of additional constraint was analyzed in \cite{cauvi2025parameterized}. As an application  of these relabelling problems, one can consider a transportation network modeled by a temporal graph $\mathcal{G}$ and evaluate how to transform the transportation time tables for the transportation network to satisfy temporal transitivity.

\bibliographystyle{plain}
\small{
    \bibliography{references}
}
\newpage

\section{Appendix}

We propose a first translation to Ghouila-Houri's article \cite{GhouilaHouri}, where the results cited and extended in our work are presented. The original name of the article is \textit{Caractérisation des graphes non orientés dont
on peut orienter les arêtes de manière à obtenir le graphe d'une
relation d'ordre}.

\subsection{Ghouila-Houri - Characterization of undirected graphs that can be oriented as a partial order}

Graphs of order relations have interesting properties that do not depend on the orientation of edges (chromatic number, stability, etc.). It is therefore useful to recognize such graphs without knowing the orientation of their edges. C. Berge and A. J. Hoffman conjectured a characterization, which is proved here.

\vspace{2ex}

Let $G = (X, U)$ be an undirected graph. The elements of $U$, called edges of $G$, are 2-element subsets of the vertex set $X$. To orient $G$ means choosing, for each edge $\{a,b\}$, exactly one of the ordered pairs $(a,b)$ or $(b,a)$. Let $V$ be the set of chosen ordered pairs. We say that the orientation $V$ is \emph{transitive} if:

\[
(a,b) \in V \text{ and } (b,c) \in V \Rightarrow (a,c) \in V.
\]

$V$ is \emph{quasi-transitive} if:

\[
(a,b) \in V \text{ and } (b,c) \in V \Rightarrow \{a,c\} \in U.
\]

Any transitive orientation is, in particular, quasi-transitive.

\begin{lemma}
If there exists a quasi-transitive orientation of $G$, then there also exists a transitive orientation of $G$.
\end{lemma}

\begin{proof}
We assume the lemma holds for $|X| < n$, and consider a graph $G = (X,U)$ with $|X| = n$ that admits a quasi-transitive orientation $V$. Observe that if three vertices $a,b,c$ satisfy $(a,b) \in V, (b,c) \in V, (c,a) \in V$, then any other vertex of $G$ is adjacent to 0, 2, or 3 of these vertices.

If $V$ is not transitive, then there exist three vertices $x_1, x_2, x_3$ such that $(x_1,x_2) \in V,(x_2,x_3) \in V,(x_3,x_1) \in V$. Two cases arise.

\paragraph{Case 1.}
Every vertex different from $x_1,x_2,x_3$ that is adjacent to one of them is adjacent to all three.

Remove $x_2$ and $x_3$, and orient the remaining graph transitively (this is possible by induction). Then orient each edge $\{x,x_2\}$ or $\{x,x_3\}$ (for $x \notin \{x_1,x_2,x_3\}$) in the same direction as $\{x,x_1\}$, and orient the triangle formed by $x_1,x_2,x_3$ transitively. This yields a transitive orientation of $G$.

\paragraph{Case 2.}
There exists at least one vertex different from $x_1,x_2,x_3$ adjacent to two of them but not the third; for example, adjacent to $x_2$ and $x_3$, but not to $x_1$. Let $A$ be the set of vertices different from $x_2$ and $x_3$ adjacent to both $x_2$ and $x_3$, let $W$ be the set of pairs of vertices in $A$ that are not edges of $G$. Let $C$ be the connected component containing $x_1$ in the graph $(A,W)$; then $|C| > 1$.

Every $c \in C$ satisfies $(c,x_2) \in V $ and $(x_3,c) \in V$. We show that any vertex not in $C$ that is adjacent to one vertex of $C$ is necessarily adjacent to all vertices of $C$. Indeed, let $x \notin C$ and $a,b \in C$. If $x$ is adjacent to $a$, then it must be adjacent to $x_2$ or $x_3$, hence $x$ also adjacent to $b$; otherwise we get $x \in A$ and $\{x,b\} \in W$, implying $x \in C$, a contradiction.

It is now sufficient to orient transitively the subgraph induced by $C$ first, then the graph obtained by removing all vertices of $C$ except $x_1$ and to orient every edge $\{x,x'\}$ with $x \notin C$ and $x' \in C$ in the same direction as $\{x,x_1\}$, to obtain a transitive orientation of $G$.
$\qed$\end{proof}

\begin{theorem}
Let $G = (X,U)$ be an undirected graph, and let $\overline{G}$ the undirected graph whose vertices are the ordered pairs $(a,b)$ such that $\{a,b\} \in U$ and by linking every vertex $(a,b)$ to vertex $(b,a)$ and to every vertex $(b,c)$ such that $\{a,c\} \notin U$. Then $G$ admits a transitive orientation if and only if $\overline{G}$ is bipartite (i.e., has no odd cycle).
\end{theorem}

\begin{proof}
\begin{enumerate}
\item If $V$ is a transitive orientation of $G$, then $V$ is an independent set in $\overline{G}$, and the complement of $V$ is also an independent set in $\overline{G}$. Hence $\overline{G}$ is bipartite.

\item Conversely, if $\overline{G}$ is 2-colorable, say red and black, then the set $V$ of red vertices defines a quasi-transitive orientation of $G$. By the previous lemma, $G$ then admits a transitive orientation.
\end{enumerate}
$\qed$\end{proof}

\noindent (*) Session of February 12, 1962.

\subsection{TTO characterization via temporal ordered patterns}
\label{appendix:tto-pattern-characterization}

Proof of Theorem \ref{theo:tto-pattern-characterization}: \textit{The temporal graph $\mathcal{G}$ admits a temporal transitive orientation if and only if it
belongs to $\mathcal{C}_{TTO}$.}

\begin{proof}
\vspace{2ex}

($\Rightarrow$) Let $O$ be a TTO of $\mathcal{G}=(G, \lambda)$. By Lemma \ref{lem:qtto-monolabel-triangle}, $O$ is acyclic and therefore defines a partial order $\sigma$ of $V$. Let $\sigma'$ be a linear extension of $\sigma$. Then, for every arc $(u,v) \in O$, $u \prec_{\sigma'} v$. Consider the ordered triplet $a,b,c$ and let $t_{1}$, $t_{2}$ and $t_{3}$ be the time-labels associated to the potential edges $ab$, $bc$ and $ac$ respectively. Note that the only possible arcs between the triplet's nodes are $(a,b), (b,c)$ and $(a,c)$. As the orientation was TTO, if arcs $(a,b), (b,c) \in O$ such that $t_1 \leq t_2$ then necessarily $(a,c) \in O$ such that $t_3 \geq t_2$. We conclude that the pattern is avoided by every ordered triplet $a,b,c$.

\vspace{2ex}
($\Leftarrow$) Let $\sigma'$ be an ordering avoiding the pattern. We know that for every triplet $a,b,c$ ordered by $\sigma'$, if $ab \in E$ and $bc \in E$ then either $\lambda(ab) > \lambda(bc)$ or $ac \in E$ such that $\lambda(ac) \geq \lambda(bc)$. Then, if we orient all edges from left to right according to $\sigma'$, we obtain a TTO of $\mathcal{G}$.

\end{proof}

\subsection{Strict TTO characterization via temporal ordered patterns}
\label{appendix:strict-tto-pattern-characterization}

Let $O$ be a Strict TTO of a temporal graph $\mathcal{G}=(G,\lambda)$, with $G=(V,E)$. Note that unlike temporal transitive orientations, $O$ can contain a directed cycle. However, directed cycles must be monolabel, and even more, we state the following.

\begin{obs}
All strongly connected components in $O$ are monolabel.
\label{obs:strongly-connected-comp-labels-strict} 
\end{obs}

\begin{proof}
All directed cycles in $O$ are necessarily monolabel. Suppose this is not true and let $C$ be the smallest non-monolabel directed cycle. Note that $|C| > 3$ as otherwise $O$ would not be Strict TTO. By hypothesis, there must exist two consecutive arcs in $C$ forming a strictly time-respecting directed path. But this implies the existence of a chord between these two arcs creating a strictly smaller directed cycle, a contradiction. Let $S$ be a strongly connected component of $O$. Then, each arc induced by $S$ will belong to a monolabel directed cycle of vertices all contained in $S$. Consider two directed cycles $C_{1}$ and $C_{2}$ both belonging to $S$. Note that the existence of $C_{1}$ and $C_{2}$ in $S$ of different time-labels implies the existence of two directed cycles $C_{1}'$ and $C_{2}'$ in $S$ of different time-labels $t_1$ and $t_2$ respectively and sharing a common vertex $v$. Let $t_1 < t_2$. Then, there necessarily exists an arc $(v_1,v)$ belonging to $C_1'$ and an arc $(v,v_2)$ belonging to $C_2'$, of time-labels $t_1$ and $t_2$ respectively. As these two arcs form a strictly time-respecting directed path, necessarily there exists a third arc $(v_1,v_2)$ in $O$. But then, $S$ induces a directed cycle $C_1 \cup C_2 \cup \{(v_1,v_2)\} \setminus \{(v_1,v),(v,v_2)\}$ with two distinct time-labels which implies a contradiction to directed cycles being monolabel in all Strict TTO.
\end{proof}

The following straightforward observations, related to the incoming and outcoming arcs of strongly connected components in $O$, will be useful when proposing a characterization via temporal ordered patterns, in Theorem \ref{theo:strict-tto-pattern-characterization}.

\begin{obs}
Let $S$ be a strongly connected component of $O$ with time-label $t_{1}$ and let $(v,w)$ be an outcoming arc from $S$ such that $v \in S, w \notin S$, $\lambda(vw) = t_2$ and $t_{1} < t_{2}$. Then, for every other vertex $u \in S$, $(u,w) \in O$ with $\lambda(uw) = t_{2}$.
\label{obs:connected-component-outcoming-arcs}
\end{obs}

\begin{obs}
Let $S$ be a strongly connected component of $O$ with time-label $t_{2}$ and let $(v,w)$ be an incoming arc to $S$ such that $v \notin S, w \in S$, $\lambda(vw) = t_{1}$ and $t_{1} < t_{2}$. Then, for every other vertex $w' \in S$ such that $(w,w') \in O$, $(v,w') \in O$ with $\lambda(vw') \geq t_{2}$.
\label{obs:connected-component-incoming-arcs}
\end{obs}

Proof of Theorem \ref{theo:strict-tto-pattern-characterization}: \textit{The temporal graph $\mathcal{G}$ admits a strict temporal transitive orientation if and only if it belongs to $\mathcal{C}_{Strict-TTO}$.}

\begin{proof}

$\rightarrow$)
Let $O$ be a Strict TTO of $\mathcal{G}$ and let $D'$ be the directed acyclic graph obtained by contracting its strongly connected components. Let $\sigma'$ be the pre-order on the vertices of $\mathcal{G}$ defined by $D'$ and take $\sigma$ as a linear extension of $\sigma'$ such that all strongly connected components in $O$ are consecutive in $\sigma$. Observe that all arcs $(u,v)$ between different strongly connected components will be oriented from left to right by $\sigma$, i.e. $u \prec_{\sigma} v$. Let $a$,$b$ and $c$ be a triplet of vertices in $\mathcal{G}$ ordered by $\sigma$. Note that if $\lambda(ab) < \lambda(bc)$ but $(a,c) \notin E$ or $\lambda(ac) < \lambda(bc)$, then necessarily there exists a strongly connected component $S$ in $O$ such that either $a,b \in S$ with $c \notin S$ or $b,c \in S$ with $a \notin S$, by Observation \ref{obs:connected-component-outcoming-arcs} we claim that only the second case is possible. Given five ordered vertices $a$, $b$, $c$, $d$ and $e$, we will analyse the conditions for each pattern to arise and conclude that all patterns are avoided by $\sigma$.

\begin{itemize}
\item \underline{Patterns 1 and 2:}
Suppose pattern 1 is produced by the set of ordered vertices $a$, $b$, $c$, $d$. We deduce vertices $b$ and $d$ must belong to the same strongly connected component $S$ in $O$. As vertices belonging to a given strongly connected component should be consecutive in $\sigma$, we state that $c$ also belongs to $S$. However, we reach a contradiction by Observation \ref{obs:strongly-connected-comp-labels-strict} since $t_2 \neq t_3$. For pattern 2, the argument is analogous.

\item \underline{Pattern 3:}
Assume vertices $a$, $b$, $c$, $d$ produce pattern 3. Again, vertices $b$ and $c$ must belong to the same strongly connected component $S$ in $O$. Additionally, since $t_2 \neq t_3$, we state that vertex $d \notin S$. Then, arc $(b,d)$ is an outcoming arc from $S$ such that $t_2 < t_3$. But since $(c,d) \notin E$ or $\lambda(cd) < t_3$ this contradicts Observation \ref{obs:connected-component-outcoming-arcs}.

\item \underline{Pattern 4:}
Suppose vertices $a$, $b$, $c$, $d$ form pattern 4. Then, there must exist a strongly connected component $S$ in $O$ such that $b,c \in S$ and since $t_2 \neq t_3$, by Observation \ref{obs:strongly-connected-comp-labels-strict}, $a \notin S$. Then, $(a,c) \in O$ but this contradicts Observation \ref{obs:connected-component-incoming-arcs}, as $(a,b)$ should belong to $O$ with $\lambda(ab) \geq t_3$. 

\item \underline{Pattern 5:}
Let the ordered vertices $a$, $b$, $c$, $d$, $e$ form pattern 5. We argue the existence of two strongly connected components $S_1$ and $S_2$ in $O$ such that $b,c \in S_1$ and $c,d \in S_2$. Since $c$ belongs both to $S_1$ and $S_2$, necessarily $S_1 = S_2$. But then, by Observation \ref{obs:strongly-connected-comp-labels-strict}, $t_3$, $t_4$ and $t_5$ should be equal, which contradicts the hypothesis.

\end{itemize}

$\leftarrow$) 

Let $\sigma$ be an ordering of the vertices in $\mathcal{G}$ avoiding the patterns, we construct an orientation $O$ as follows. Given a triplet $a , b, c$ ordered by $\sigma$ with $ab,bc \in E$, if $\lambda(ab) = t_1$, $\lambda(bc) = t_2$ with $t_1 < t_2$, and either $ac\notin E$ or $\lambda(ac) < t_2$, then we add arc $(c,b)$ to $O$. We orient every other edge from left to right according to $\sigma$. Let's verify $O$ is a Strict TTO of $\mathcal{G}$. That is, taking $t_1$, $t_2$ and $t_3$ such that $t_1 < t_2$ and $t_3 \geq t_2$, none of the following cases occur, for any triplet $a$, $b$, $c$ ordered by $\sigma$ and oriented by $O$.

\begin{itemize}
\item \underline{Case 1:} $(a,b), (b,c) \in O$ with either $ac \notin E$, $\lambda(ac) < t_2$ or $(c,a) \in O$

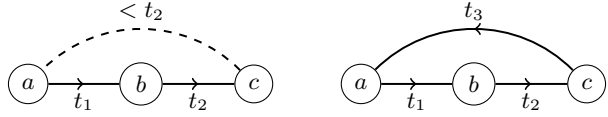
\begin{figure}[H]
    \begin{center}
        \begin{tikzpicture}[scale=0.75]\label{2patterns}
        \node (a) at (0, 0) [circle, draw] {$a$};
        \node (b) at (2, 0) [circle, draw] {$b$};
        \node (c) at (4, 0) [circle, draw] {$c$};
        \draw[mid arrow] (a) -- (b) node[midway, below] {$t_1$};
        \draw[mid arrow] (b) -- (c) node[midway, below] {$t_2$};
        \draw[-,  line width=0.75pt, bend right=45, dashed] (c) to node[midway, above] {$< t_2$} (a);
        \end{tikzpicture} \hspace{5ex}
        \begin{tikzpicture}[scale=0.75]
        \node (a) at (0, 0) [circle, draw] {$a$};
        \node (b) at (2, 0) [circle, draw] {$b$};
        \node (c) at (4, 0) [circle, draw] {$c$};
        \draw[mid arrow] (a) -- (b) node[midway, below] {$t_1$};
        \draw[mid arrow] (b) -- (c) node[midway, below] {$t_2$};
        \draw[mid arrow, bend right=45] (c) to node[midway, above] {$t_3$} (a);
        \end{tikzpicture}
    \end{center}
\end{figure}
    
If both $(a,b), (b,c) \in O$, then $ac \in E$ such that $\lambda(ac) \geq t_2$. Otherwise, the orientation construction rules would have fixed $(c,b) \in O$. Suppose now $(c,a) \in O$ with time label set $t_3$, this implies the existence of vertex $a'$ such that $a'a \in E$ with $\lambda(a'a) < t_3$ , and either $a'c \notin E$ or $\lambda(a'c) < t_3$. But then vertices $a'$, $a$, $b$, $c$ would form pattern 1, which is a contradiction.
    
\vspace{1ex}
    
\item \underline{Case 2:} $(b,a),(c,b) \in O$ with either $ac \notin E$, $\lambda(ac) < t_2$ or $(a,c) \in O$
    
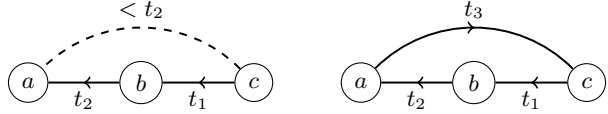
\begin{figure}[H]    
    \begin{center}
        \begin{tikzpicture}[scale=0.75]
        \node (a) at (0, 0) [circle, draw] {$a$};
        \node (b) at (2, 0) [circle, draw] {$b$};
        \node (c) at (4, 0) [circle, draw] {$c$};
        \draw[mid arrow] (b) -- (a) node[midway, below] {$t_2$};
        \draw[mid arrow] (c) -- (b) node[midway, below] {$t_1$};
        \draw[-, line width=0.75pt, bend right=45, dashed] (c) to node[midway, above] {$< t_2$} (a);
        \end{tikzpicture} \hspace{5ex}
        \begin{tikzpicture}[scale=0.75]
        \node (a) at (0, 0) [circle, draw] {$a$};
        \node (b) at (2, 0) [circle, draw] {$b$};
        \node (c) at (4, 0) [circle, draw] {$c$};
        \draw[mid arrow] (b) -- (a) node[midway, below] {$t_2$};
        \draw[mid arrow] (c) -- (b) node[midway, below] {$t_1$};
        \draw[mid arrow, bend left=45] (a) to node[midway, above] {$t_3$} (c);
        \end{tikzpicture}
    \end{center} 
\end{figure}
    
If $(b,a),(c,b) \in O$, we can state the existence of both vertex $a'$ such that $a'a \in E$ with $\lambda(a'a) < t_2$, and either $a'b \notin E$ or $\lambda(a'b) < t_2$, and vertex $b'$ such that $b'b \in E$ with $\lambda(b'b) < t_1$, and either $b'c \notin E$ or with $\lambda(b'c) < t_1$. If $ac \notin E$, $\lambda(ac) < t_2$ or $(a,c) \in O$, then vertices $a'$, $a$, $b$, $c$ would form pattern 5.
    
\vspace{1ex}    
    
\item \underline{Case 3:} $(b,a),(a,c) \in O$ with either $bc \notin E$, $\lambda(bc) < t_2$ or $(c,b) \in O$
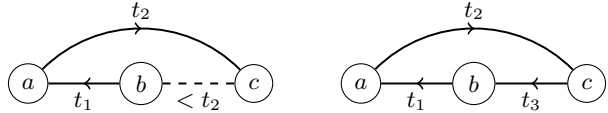
\begin{figure}[H]
    \begin{center}
        \begin{tikzpicture}[scale=0.75]
        \node (a) at (0, 0) [circle, draw] {$a$};
        \node (b) at (2, 0) [circle, draw] {$b$};
        \node (c) at (4, 0) [circle, draw] {$c$};
        \draw[mid arrow] (b) -- (a) node[midway, below] {$t_1$};
        \draw[-, line width=0.75pt, dashed ] (b) -- (c) node[midway, below] {$< t_2$};
        \draw[mid arrow, bend left=45] (a) to node[midway, above] {$t_2$} (c);
        \end{tikzpicture} \hspace{5ex}
        \begin{tikzpicture}[scale=0.75]
        \node (a) at (0, 0) [circle, draw] {$a$};
        \node (b) at (2, 0) [circle, draw] {$b$};
        \node (c) at (4, 0) [circle, draw] {$c$};
        \draw[mid arrow] (b) -- (a) node[midway, below] {$t_1$};
        \draw[mid arrow] (c) -- (b) node[midway, below] {$t_3$};
        \draw[mid arrow, bend left=45] (a) to node[midway, above] {$t_2$} (c);
        \end{tikzpicture}
    \end{center}
\end{figure}

If $(b,a) \in O$, there must exist a vertex $a'$ such that $a'a \in E$, with $\lambda(a'a) < t_1$, and either $a'b \notin E$ or $\lambda(a'b) < t_1$, this corresponds to pattern 3. If $(c,b)\in O$, then vertices $a$, $b$, $b'$, $c$ would form pattern 5.

\vspace{1ex}

\item \underline{Case 4:} $(c,a),(a,b) \in O$ with either $bc \notin E$, $\lambda(bc) < t_2$ or $(b,c) \in O$
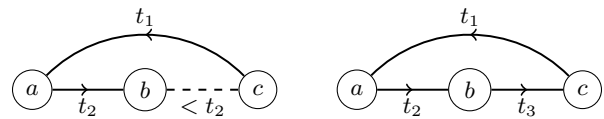
\begin{figure}[H]
    \begin{center}
        \begin{tikzpicture}[scale=0.75]
        \node (a) at (0, 0) [circle, draw] {$a$};
        \node (b) at (2, 0) [circle, draw] {$b$};
        \node (c) at (4, 0) [circle, draw] {$c$};
        \draw[mid arrow] (a) -- (b) node[midway, below] {$t_2$};
        \draw[-, line width=0.75pt, dashed ] (b) -- (c) node[midway, below] {$< t_2$};
        \draw[mid arrow, bend right=45] (c) to node[midway, above] {$t_1$} (a);
        \end{tikzpicture}\hspace{5ex}
        \begin{tikzpicture}[scale=0.75]
        \node (a) at (0, 0) [circle, draw] {$a$};
        \node (b) at (2, 0) [circle, draw] {$b$};
        \node (c) at (4, 0) [circle, draw] {$c$};
        \draw[mid arrow] (a) -- (b) node[midway, below] {$t_2$};
        \draw[mid arrow] (b) -- (c) node[midway, below] {$t_3$};
        \draw[mid arrow, bend right=45] (c) to node[midway, above] {$t_1$} (a);
        \end{tikzpicture}
    \end{center}
\end{figure}

As for the previous case, orientation $(c,a)$ implies the existence of a vertex $a'$ such that $a'a \in E$ such that either $a'c \notin E$ or $\lambda(a'c) < t_1$. As $t_1 \neq t_2$, then vertices $a'$, $a$, $b$, $c$ reproduce pattern 1. If $(b,c) \in O$, the vertices will still produce the first pattern. 

\vspace{1ex}

\item \underline{Case 5:} $(a,c),(c,b) \in O$ with either $ab \notin E$, $\lambda(ab) < t_2$ or $(b,a) \in O$
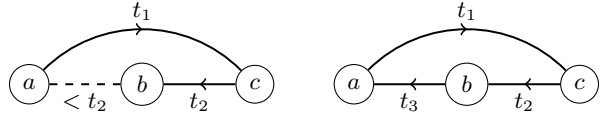
\begin{figure}[H]    
    \begin{center}
        \begin{tikzpicture}[scale=0.75]
        \node (a) at (0, 0) [circle, draw] {$a$};
        \node (b) at (2, 0) [circle, draw] {$b$};
        \node (c) at (4, 0) [circle, draw] {$c$};
        \draw[-, line width=0.75pt, dashed] (a) -- (b) node[midway, below] {$< t_2$};
        \draw[mid arrow] (c) -- (b) node[midway, below] {$t_2$};
        \draw[mid arrow, bend left=45] (a) to node[midway, above] {$t_1$} (c);
        \end{tikzpicture}\hspace{5ex}
        \begin{tikzpicture}[scale=0.75]
        \node (a) at (0, 0) [circle, draw] {$a$};
        \node (b) at (2, 0) [circle, draw] {$b$};
        \node (c) at (4, 0) [circle, draw] {$c$};
        \draw[mid arrow] (b) -- (a) node[midway, below] {$t_3$};
        \draw[mid arrow] (c) -- (b) node[midway, below] {$t_2$};
        \draw[mid arrow, bend left=45] (a) to node[midway, above] {$t_1$} (c);
        \end{tikzpicture}
    \end{center}
\end{figure}

As $(c,b) \in O$, there must exist a vertex $b'$ such that $b'b \in E$, with $\lambda(b',b) < t_2$ and either $b'c \notin E$ or $\lambda(b',c) < t_2$, this corresponds to pattern 4. Now, suppose $(b,a) \in O$, this corresponds to pattern 5.

\vspace{1ex}

\item \underline{Case 6:} $(b,c),(c,a) \in O$ with either $ab \notin E$, $\lambda(ab) < t_2$ or $(a,b) \in O$ 
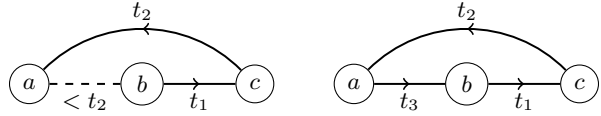
\begin{figure}[H]
    \begin{center}
        \begin{tikzpicture}[scale=0.75]
        \node (a) at (0, 0) [circle, draw] {$a$};
        \node (b) at (2, 0) [circle, draw] {$b$};
        \node (c) at (4, 0) [circle, draw] {$c$};
        \draw[-, line width=0.75pt, dashed] (a) -- (b) node[midway, below] {$< t_2$};
        \draw[mid arrow] (b) -- (c) node[midway, below] {$t_1$};
        \draw[mid arrow, bend right=45] (c) to node[midway, above] {$t_2$} (a);
        \end{tikzpicture}\hspace{5ex}
        \begin{tikzpicture}[scale=0.75]
        \node (a) at (0, 0) [circle, draw] {$a$};
        \node (b) at (2, 0) [circle, draw] {$b$};
        \node (c) at (4, 0) [circle, draw] {$c$};
        \draw[mid arrow] (a) -- (b) node[midway, below] {$t_3$};
        \draw[mid arrow] (b) -- (c) node[midway, below] {$t_1$};
        \draw[mid arrow, bend right=45] (c) to node[midway, above] {$t_2$} (a);
        \end{tikzpicture}
    \end{center}
\end{figure}
        
We deduce the existence of a vertex $a'$ such that $a'a \in E$, with $\lambda(a',a) < t_2$ and either $a'c \notin E$ or $\lambda(a',c) < t_2$,  this corresponds to pattern 2. If $(a,b) \in O$, the vertices will produce pattern 1.

\end{itemize}

For each triplet $a$, $b$, $c \in V$ ordered by $\sigma$, we have analysed all six scenarios were $O$ could fail to satisfy the Strict TTO constraints and proved the impossibility of them to occur. Therefore, we can state that if $(a,b) \in O$ and $(b,c) \in O$ with $\lambda(ab) < \lambda(bc)$ then by construction of $O$ we ensure that $(a,c)$ will also belong to the orientation with a time-label such that $\lambda(ac) \geq \lambda(bc)$. We conclude that $O$ is a Strict TTO of $\mathcal{G}$.
\end{proof}

\subsection{Multilabel temporal transitive orientations}
\label{appendix:mtto}

We start by observing that orienting an edge of a multilabel temporal graph will fix the orientation for all its temporal time-labels. Let's now recall the definition of multilabel temporal transitive orientations.

\begin{definition}
Let $\mathcal{G}=(G,\lambda)$ be a multilabel temporal graph and $O$ an orientation of $\mathcal{G}$. We will say $O$ is a multilabel temporal transitive orientation (MTTO) if whenever $(a,b),(b,c) \in O$ with $\lambda(ab) \preceq \lambda(bc)$, then necessarily $(a,c) \in O$ with $\lambda(ac) \hspace{0.5ex} \cancel{\prec} \hspace{0.5ex} \lambda(bc)$.
\end{definition}

An alternative way to put it is that whenever $(a,b),(b,c) \in O$ with $min\{\lambda(ab)\} \leq max\{\lambda(bc)\}$, then $(a,c) \in O$ with $min\{\lambda(ac)\} \geq max\{\lambda(bc)\}$, taking as $min$ and $max$ the minimum and maximum elements of a given set. As for simple temporal graphs, one can define the Strict MTTO version by asking for the third arc $(a,c)$ to belong to $O$ only when $\lambda(ab) \prec \lambda(bc)$. One can verify that these definitions correspond to an extension of the original definitions for temporal transitivity by considering simple temporal graphs. Also, as for the static case, MTTO are acyclic and for the strict version, only monolabel strongly connected components having a singleton time-label set can be non-trivial. Next, we present some arguments to state that the characterization and algorithmic result presented for simple temporal graphs hold for the general case of multilabel temporal graphs as well. Complete proofs are omitted as they are analogous to the simple case.

An analogous definition for \emph{quasi-multilabel temporal transitive orientations} can be proposed as follows. 

\begin{definition}
Let $\mathcal{G}=(G,\lambda)$ be a multilabel temporal graph and $O$ an orientation of $\mathcal{G}$. We will say $O$ is a quasi-multilabel temporal transitive orientation (QMTTO) if whenever $(a,b),(b,c) \in O$ with $\lambda(ab) \preceq \lambda(bc)$, then necessarily $(a,c) \in O$ or $(c,a) \in O$ with $\lambda(ac) \hspace{0.5ex} \cancel{\prec} \hspace{0.5ex} \lambda(bc)$.
\end{definition}

In the multilabel scenario, the equivalent to monolabel triangles considered in Section \ref{sect:tto-algorithm} will be monolabel triangles with singleton time-label sets, as these will be the only directed triangles admitted by a QMTTO. Then, we can define the extension to correlated monolabel triangles as a quadruple of vertices $(a,b,c,d)$ of $\mathcal{G}$ such that:

\begin{itemize}
    \item triangle $abc$ is monolabel with singleton time-label set $T$
    \item $bd \in E$ with $\lambda(bd)\preceq T$ 
    \item $cd \in E$ with $\lambda(cd) \hspace{1ex} \cancel{\preceq} \hspace{1ex} T$
    \item if $(a,d) \in E$, then $\lambda(ad) \prec T$
\end{itemize}

\begin{figure}[H]
\begin{center}
\begin{tikzpicture}[scale=0.75]

\node[draw, circle, inner sep=1pt] (D) at (0,0) {$d$};
\node[draw, circle, inner sep=1pt] (B) at (2,0) {$b$};
\node[draw, circle, inner sep=1pt] (A) at (1,-1.5) {$a$};
\node[draw, circle, inner sep=1pt] (C) at (3,-1.5) {$c$};

\draw[-,line width=0.75pt] (D) -- (B) node[midway, above] {$\preceq T$};
\draw[-, dashed] (D) -- (A) node[midway, left] {$ \prec T$};
\draw[-, line width=0.75pt] (C) -- (A) node[midway, below] {$T$};
\draw[-, line width=0.75pt] (C) -- (B) node[midway, right] {$T$};
\draw[-, line width=0.75pt] (A) -- (B) node[midway, left] {$T$};
\draw[line width=0.75pt] (D) .. controls (-1.8,-2) and (2.1,-3.7) .. (C) node[midway, below left] {$ \hspace{1ex} \cancel{\preceq} \hspace{1ex} T$};

\end{tikzpicture}
\end{center}
\caption{Correlated monolabel triangle $(a,b,c,d)$}
\label{fig:bad-quadruple-mtto}
\end{figure}
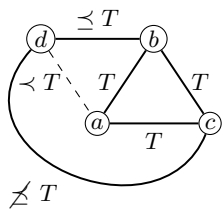

One can prove the same forcings as in the static case will be produced for a given correlated monolabel triangle in the multilabel setting. Therefore, we can consider an analogous definition of \emph{almost-multilabel temporal transitive orientations} (AMTTO) leading to an equivalent characterization and lexicographic strategy. Note that only the minimum and maximum values of the time-label sets are considered when formulating the constraints. Then, if we considered multilabel temporal graphs with ordered time-label sets, the size of the label sets will not impact the algorithm's complexity. These arguments lead to the formulation of Theorem \ref{theo:mtto-gouila-houri-generalization}.

From \cite{casteigts2024simple}, a multilabel temporal graph $\mathcal{G}$ is \textit{proper} if $\lambda(e) \hspace{0.25ex} \cap \hspace{0.25ex} \lambda(e') = \emptyset$ whenever $e$ and $e'$ are incident to a common vertex. Then, we can state the following, which is particularly interesting for the Strict MTTO recognition problem, proved NP-hard for simple temporal graphs.

\begin{theorem}
The MTTO and Strict MTTO recognition problem for proper multilabel temporal graphs can be solved in $O(nm)$.
\label{theo:recognition-proper-temporal-graphs}
\end{theorem}

\begin{proof}
No proper graph can contain a triangle such that two of its arcs share the same time-label set. Then, the constraints for both MTTO and Strict MTTO will be entirely captured by $Imp(\mathcal{G})$. Deciding whether $Imp(\mathcal{G})$ is satisfiable can be solved in $O(nm)$ which is the cost of computing the implication graph, as well as its size. 
\end{proof}

\end{document}